\documentclass[11pt,a4paper]{amsart}
\usepackage{hyperref}
\usepackage{graphicx} 
\usepackage{wasysym}
\usepackage[mathscr]{eucal}

\usepackage{amsmath}

\usepackage{amssymb}
\usepackage{amsthm}
\usepackage{amscd}

\usepackage{amsfonts}
\usepackage{tikz}
\usepackage{enumerate}

\usepackage[english]{babel}

\newtheorem{theorem}{Theorem}[section]

\newtheorem{proposition}[theorem]{Proposition}
\newtheorem{corollary}[theorem]{Corollary}

\theoremstyle{definition}
\newtheorem{definition}[theorem]{Definition}
\newtheorem{example}[theorem]{Example}

\newtheorem{remark}[theorem]{Remark}

\newcommand{\V}{{\scriptscriptstyle \mathscr{V}}}
\newcommand{\W}{{\scriptscriptstyle \mathscr{W}}}

\makeatletter
\newcommand{\twoheaddownarrow}{\mathrel{\mathpalette\thd@arrow{}}}
\newcommand{\thd@arrow}[2]{%
	\rotatebox[origin=c]{90}{$\m@th#1\twoheadleftarrow$}%
}

\newcommand{\twoheaduparrow}{\mathrel{\mathpalette\thd@arrowd{}}}
\newcommand{\thd@arrowd}[2]{%
	\rotatebox[origin=c]{90}{$\m@th#1\twoheadrightarrow$}%
}
\makeatother

\begin{document}

	
	\title[]{Topological preservation under aggregation of quantale-valued metric spaces}
	
	\author[A. Fructuoso-Bonet]{Alejandro Fructuoso-Bonet}
	\address[A. Fructuoso-Bonet]{Programa de Doctorado en Matem\'aticas, Instituto Universitario de Matem\'atica Pura y Aplicada, Universitat Polit\`ecnica de Val\`encia, Camino de Vera s/n, 46022 Valencia, Spain,\newline    Florida Universitaria, C/ Rei en Jaume I, nº 2, 46470 Catarroja (Valencia), Spain}
	
	\email{afrubon@posgrado.upv.es}
	\author[J. Rodr\'{\i}guez-L\'opez]{Jes\'us Rodr\'{\i}guez-L\'opez}
	\address[J. Rodr\'{\i}guez-L\'opez]{Instituto Universitario de Matem\'atica Pura y Aplicada, Universitat Polit\`ecnica de Val\`encia, Camino de Vera s/n, 46022 Valencia, Spain}
	\thanks{Last author's research
		is part of the project PID2022-139248NB-I00 funded by MICIU/AEI/10.13039/501100011033 and ERDF/EU}
	\email{jesus.rodriguezlopez@upv.es}

	

	\begin{abstract}
		Aggregation functions provide a general mechanism for combining multiple inputs into a single output and constitute a fundamental tool in information fusion. Beyond the classical numerical setting, aggregation can also be considered at the level of mathematical structures, where the main question is to determine which properties of the original structures are preserved by the aggregation process. In this paper, we study the preservation of topology within a unified framework based on continuous quantales and quantale-enriched categories. We introduce strongly preserving functions, for which the natural topology induced by a $\mathscr V$-metric space coincides with that induced after aggregation. We decompose this property into its upper and lower versions and show that upper strong preservation is equivalent to Scott continuity at the unit of the domain quantale, whereas lower strong preservation is characterized in terms of the fibre of the unit and the Scott upper semicontinuity of the inverse multifunction. Combining these results, we obtain a complete characterization of strongly preserving functions, together with analogous results for separated, symmetric, and separated and symmetric $\mathscr V$-metric spaces. Finally, we specialize the general theory to functions from $\mathsf P_+^{\,I}$ to $\mathsf P_+$, obtaining characterizations of strongly extended quasi-pseudometric and pseudometric aggregation functions on Cartesian products and on common underlying sets, together with corresponding results for extended quasi-metric and metric aggregation functions on Cartesian products.
	\end{abstract}    
	
	
	\date{\today}
	\maketitle
	
	\section{Introduction}
	
	Aggregation functions are usually understood as mappings that combine several numerical inputs into a single numerical output \cite{GrabischAlter11,XuDa03}. They constitute one of the basic mathematical tools for information fusion \cite{DuboisPrade04,BustinceAlter26,Alcantud26}.
	Beyond this classical numerical setting, however, aggregation procedures can also be
	used to combine more complex mathematical structures. In this case, the
	problem is to determine those functions that preserve the defining
	properties of the structures under consideration. This point of view has
	been studied for se\-veral kinds of generalized distances and related
	structures, including metrics and quasi-metrics, fuzzy metrics,
	indistinguishability operators, and metric modulars
	\cite{BorsikDobos81b,Dobos98,FBRL25a,FBRL25b,MayorVale10,MayorVale19,BFVale24,BFVale26,PRLVale21,MassaVale13,MinyaVale19}.

	A classical example is provided by the aggregation of metrics on
	Cartesian products of metric spaces. Let $I$ be a nonempty index set and let
	\[F:[0,+\infty)^I\longrightarrow[0,+\infty).\]
	Given a family of metric spaces $\big\{(X_i,d_i):i\in I\big\}$, consider the map
	\[d_\Pi: \left(\prod_{i\in I}X_i\right) \times \left(\prod_{i\in I}X_i\right) \longrightarrow [0,+\infty)^I\]
	defined by
	\[d_\Pi(x,y)=\bigl(d_i(x_i,y_i)\bigr)_{i\in I}.\]
	The function $F$ is said to aggregate metrics on products if
	$F\circ d_\Pi$ is a metric on $\prod_{i\in I}X_i$ for every such
	family. These functions were introduced and studied by Borsík and Doboš
	under the name of metric preserving functions
	\cite{BorsikDobos81b,Dobos98}. Similar aggregation problems have
	subsequently been considered for pseudometrics, quasi-metrics,
	quasi-pseudometrics, fuzzy metrics, and other generalized metric
	structures \cite{BFMinyaVale23,BFVale24,FBRL25b,MayorVale10,MayorVale19,MinyaVale19,PRLVale21}.
	
	Many of these structures can be regarded as particular instances of ca\-tegories enriched over a quantale \cite{Lawv73} (see the definition in Section \ref{sec:Vmet}). For example, preordered sets are categories enriched over the two-element quantale, while extended quasi-pseudometric spaces are categories enriched over the Lawvere quantale
	\[\mathsf P_+=\bigl([0,+\infty],\leq^{\mathrm{op}},+,0\bigr).\]
	This enriched-categorical interpretation provides a common framework in
	which different aggregation problems can be studied simultaneously. Thus, rather than treating each class of aggregation functions separately, the quantale-based approach makes it possible to identify structural conditions that are common to a broad family of aggregation processes.
	
	In \cite{FBRL25a}, we introduced the notion of a preserving function between quantales. Roughly speaking, a map 
	\[F:(\mathscr V,\preceq_\V,\ast)\longrightarrow(\mathscr W,\preceq_\W,\star)\]
	is preserving if, for every $\mathscr V$-category $(X,a)$, the composite $F\circ a$ defines a $\mathscr W$-category on the same underlying set. Preserving functions were characterized as lax morphisms of quantales, and this characterization allowed several known results about the aggregation of generalized metric structures to be recovered
	within a single framework.
	
	The aggregation problem also has an important topological aspect. Indeed, the fact that an aggregated metric satisfies the axioms of a metric does not by itself guarantee that the topology it induces coincides with the product topology of the original metric spaces. Borsík and Doboš therefore considered the stronger requirement that, for every family of metric spaces $\big\{(X_i,d_i):i\in I\big\}$, the topology induced by the aggregated metric coincides with the product topo\-logy, that is,
	\[\mathscr T(F\circ d_\Pi)=\prod_{i\in I}\mathscr T(d_i).\]
	Functions satisfying this condition are called strongly
	metric preserving functions \cite{BorsikDobos81b,Dobos98}. More recently, in \cite{FBRL26a} we studied the corresponding problem
	for (quasi-)(pseudo)metric aggregation functions, while in
	\cite{FBRL27b} we considered the analogous problem for preorder
	aggregation functions, both when the original structures are defined on
	different sets and aggregated on their Cartesian product and when they
	are defined on a common underlying set. In the latter setting, the natural topology with which the topology
	of the aggregated structure must be compared is the supremum of the
	topologies induced by the original structures.
	
	The purpose of the present paper is to study this topological compatibi\-lity problem within the general framework of preserving functions between quantales. Let
	\[F:(\mathscr V,\preceq_\V,\ast)\longrightarrow(\mathscr W,\preceq_\W,\star)\]
	be a preserving function. When $\mathscr V$ and $\mathscr W$ are continuous quantales, every $\mathscr V$-metric space $(X,a)$ and every $\mathscr W$-metric space $(X,F\circ a)$ carry natural topologies, denoted respectively by $\mathscr T(a)$ and $\mathscr T(F\circ a)$ \cite{DagninoFarjudianMoggi23,CookWeiss21,Flagg97,FlaggKopp97}. It is therefore natural to ask under which conditions these two topologies coincide. Our main aim is to characterize this topological invariance purely in terms of properties of the preserving function $F$.
	
	A key point in our approach is to consider separately the two inclusions
	\[\mathscr T(F\circ a)\subseteq\mathscr T(a)\qquad\text{and}\qquad\mathscr T(a)\subseteq\mathscr T(F\circ a),\]
	which lead respectively to the notions of upper strongly preserving and lower strongly preserving functions. A preserving function is strongly preserving when both conditions hold. As will be shown, the two inclusions are determined by rather different properties: upper strong preservation is characterized by Scott continuity of $F$ at the unit $1_\V$ of the domain quantale, whereas lower strong preservation is characterized in terms of the fibre of $1_\W$ and the Scott upper semicontinuity of the inverse multifunction of $F$.
	
	The paper is organized as follows. In Section \ref{sec:Vmet}, we recall the basic notions concerning continuous lattices, the way-below relation, the Scott topology, quantales, and quantale-enriched categories. We work with continuous quantales because this class is stable under arbitrary Cartesian products and because the way-below relation has an appropriate
	finite-support behaviour on product lattices. This contrasts with Flagg's value quantales \cite{Flagg97,FlaggKopp97}, which are not, in general, closed under Cartesian products. We also recall the na\-tural topology associated with a $\mathscr V$-metric space \cite{DagninoFarjudianMoggi23} and illustrate the framework with several fundamental examples, including preordered sets and extended quasi-pseudometric spaces.
	
	Section \ref{sec:top_preserving} is devoted to the topological behaviour of preserving functions. We first introduce strongly preserving functions and their upper and lower versions. Our first main result, Theorem \ref{thm:us}, shows that a preserving function is upper strongly preserving if and only if it is Scott-continuous at $1_\V$. We then characterize lower strongly preserving functions in Theorem \ref{thm:lowerstronglypreserving}. This property holds precisely when
	\[F^{-1}(1_\W)=\{1_\V\}\]
	and the inverse multifunction
	\[F^{-1}:\mathscr W\rightrightarrows\mathscr V\]
	is Scott upper semicontinuous at $1_\W$. Combining these two characterizations, Theorem \ref{thm:strongly_preserving} provides a complete description of strongly preserving functions. Corresponding characterizations are also obtained in the separated, symmetric, and separated and symmetric cases.
	
	Finally, we apply the general theory to functions
	\[F:\mathsf P_+^{\,I}\longrightarrow\mathsf P_+.\]
	In this particular case, preserving functions correspond to extended quasi-pseudometric aggregation functions \cite[Lemma~3.7]{FBRL25a}. Our results (see Propositions
	\ref{prop:upper_preserving_equal_upper_eqpm} and
	\ref{prop:lseqpm}, and Corollary \ref{cor:strongly_ext_qpm}) yield characterizations of upper, lower, and strongly extended quasi-pseudometric and pseudometric aggregation functions, both on Cartesian products and on common underlying sets. We also obtain corresponding results for extended quasi-metric and metric
	aggregation functions on Cartesian products
	(see Proposition \ref{prop:upper_seppreserving_equal_upper_eqm},
	Corollary \ref{cor:charac_upper_eqm}, and Proposition
	\ref{prop:eqmag_implies_stronglyeqmag}). In particular, several results obtained in \cite{FBRL26a,FBRL27b} are recovered as direct consequences of the general quantale-enriched framework.

	\section{Continuous quantale-valued metric spaces}\label{sec:Vmet}
	
	Classical metric spaces take values in $[0,+\infty)$. However, a far-reaching generalization of metric spaces arises by replacing the set of nonnegative real numbers with a more general ordered algebraic structure. This line of research was strongly influenced by Lawvere’s seminal work on enriched categories \cite{Lawv73}. One of the earliest contributions in this direction is due to Trillas and Alsina (see \cite{TrillasAlsina78}). Later, Kopperman \cite{Kopp88} introduced the notion of a \emph{continuity space}, where the codomain of the metric is replaced by a \emph{value semigroup}, and proved that every topological space arises as a continuity space.
	
	Nevertheless, this correspondence is not categorical: the natural functor from the category of metric spaces to the category of topological spaces (with continuous maps as morphisms) is not full \cite{CookWeiss21}. This limitation was overcome by Flagg \cite{Flagg97}, who replaced value semigroups by \emph{value quantales}, obtaining in this setting an equivalence of categories \cite{Weiss15,CookWeiss21}.
	
	However, as we will justify later, value quantales are not stable under Cartesian products. This lack of stability becomes problematic in contexts where multi-dimensional structures are required. Since Cartesian products play a central role in our development, we adopt a different viewpoint. Instead of working with value quantales in the sense of Flagg, 
	we consider continuous quantales in the sense of \cite{DagninoFarjudianMoggi23}. This choice preserves good order-theoretic properties while ensuring stability under products, which will be essential in what follows.

	A further reason for this change of perspective concerns the role of approximation in the underlying lattice. If $\mathscr{V}$ is a value quantale, the topology associated with a $\mathscr V$-category is defined using the well-below relation, which relies on complete distributivity. By contrast, for continuous quantales, the corresponding topology is defined using the way-below relation, strongly related to the Scott topology. In product lattices, the way-below relation behaves naturally (finite-support behaviour), whereas the well-below relation becomes excessively restrictive.

	This section introduces the basic notions underlying this theory. We begin with some classical concepts from order theory. Our main references are \cite{GierzHoffKeiLawMisScott,BookQuantales}, where any undefined terminology may be found.

	\begin{definition}[\mbox{\cite[Definition I-1.1]{GierzHoffKeiLawMisScott}}]
		Let $(P,\preceq)$ be a partially ordered set. We say that $x$ is \textbf{way-below} $y,$ in symbols $x\ll y$, if for every directed subset $D$ of $P$ for which $\bigvee D$ exists, the condition $y\preceq\bigvee D$ implies that there exists $d\in D$ such that $x\preceq d$.
		
		Given $x\in P,$ we will use the following standard notation: 
		\[\twoheaduparrow x =\{y\in P:x\ll y\}\qquad \twoheaddownarrow x =\{y\in P:y\ll x\}.\]
		
	\end{definition}
	
	\begin{example}\label{ex:waybelow}\
		
		\begin{enumerate}[(1)]
			\item In the complete lattice $([0,+\infty],\leq)$, one has \[ x\ll y \quad\Longleftrightarrow\quad x<y \text{ or } x=y=0. \]
			\item Given a nonempty set $X$, the way-below relation on $(\mathcal{P}(X),\subseteq)$ is given by $A\ll B$ if and only if $A$ is a finite subset of $B$.
			\item Let $\{(L_i,\preceq_i):i\in I\}$ be a family of complete lattices. Then $\prod_{i\in I} L_i$ becomes a complete lattice under the pointwise ordering $\preceq$, that is, 
			\[x\preceq y\text{ if and only if }x_i\preceq_i y_i\text{ for all }i\in I,\]
			for all $x,y\in \prod_{i\in I} L_i.$ It is easy to check that, for  $x,y\in \prod_{i\in I} L_i,$ one has
			\[x\ll y\text{ if and only if }x_i\ll y_i\text{ and }x_i=\bot_i\text{ for all but finitely many } i,\]
			where $\bot_i$ denotes the bottom element of $L_i$.
		\end{enumerate}
	\end{example}
	
	\begin{definition}[\mbox{\cite[Definition I-1.6]{GierzHoffKeiLawMisScott}}]
		A complete lattice $(L,\preceq)$ is called \textbf{continuous} if, for every $x\in L,$ one has
		$$x=\bigvee \twoheaddownarrow x.$$
	\end{definition}
	
	We shall also use the interpolation property \cite[Theorem I-1.9]{GierzHoffKeiLawMisScott} of the way-below relation in a continuous lattice $(L,\preceq)$: whenever $u\ll v$, there exists $w\in L$ such that
	\[u\ll w\ll v.\]
	Recall also that, in a continuous lattice, the way-below relation
	satisfies
	\[u,v\ll x \quad\Longrightarrow\quad u\vee v\ll x.\]
	Hence, for every $x\in L$, the set $\twoheaddownarrow x$ is directed.
	
	\begin{definition}[\cite{DagninoFarjudianMoggi23}]
		A \textbf{continuous quantale} is a triple $(\mathscr{V},\preceq,\ast)$ such that:
		\begin{itemize}
			\item $(\mathscr{V},\preceq)$ is a complete lattice and 
			$(\mathscr{V},\preceq,\ast,\top_\V)$ is a commutative integral quantale; that is, 
			$\ast \colon \mathscr V \times \mathscr V \to \mathscr V$ is an associative and commutative binary operation with unit $1_\V=\top_\V$ the top element, and for every $u\in\mathscr V$ and every family $\{v_i:i\in I\}\subseteq \mathscr V$,
			\[u \ast \bigvee_{i\in I} v_i=\bigvee_{i\in I} (u \ast v_i);\]
			
			\item the lattice $(\mathscr{V},\preceq)$ is continuous.        
		\end{itemize}
		
	\end{definition}
	
	\begin{remark}\label{rem:right_adjoint}
		The distributivity of $\ast$ over arbitrary joins in a quantale leads to the existence of a right adjoint of the mapping $u\ast -$, for a fixed but arbitrary $u\in\mathscr{V}$. The right-adjoint $u\searrow -$ is given by
		\[\; u \searrow v 
		\;=\; \bigvee \bigl\{\, w \in \mathscr{V}:u \ast w \preceq v \,\bigr\},
		\]
		for every $v\in\mathscr{V}$ \cite{BookQuantales}.
	\end{remark}
	
	We provide some standard examples of continuous quantales.
	
	\begin{example}\label{ex:continuous_quantales}\
		
		\begin{enumerate}[(1)]
			\item If
			\[\mathbf 2=\{\bot<\top\},\]
			then $(\mathbf 2,\leq,\wedge)$
			is a commutative integral continuous quantale. 
			
			More generally, every finite distributive lattice, endowed with the tensor product given by the meet, is a commutative integral continuous quantale.
			
			\item $\mathsf P_+=\big([0,+\infty],\leq^{\text{op}},+\big)$ is a commutative integral continuous quantale. Its unit is $0$, which is also the top element with respect to the order $\leq^{\mathrm{op}}$.
			
			\item If $X$ is a set, then $\big(\mathcal P(X),\subseteq,\cap\big)$ is a commutative integral continuous quantale having $X$ as unit.
		\end{enumerate}
	\end{example}
	
	\begin{example}\label{ex:product_quantales}
		
		Let $\big\{(\mathscr{V}_i,\preceq_i,\ast_i) : i \in I\big\}$ 
		be an arbitrary family of continuous quantales. 
		Consider the Cartesian product $\prod_{i\in I}\mathscr{V}_i$ endowed with 
		the componentwise order (see Example \ref{ex:waybelow}) 
		and the componentwise operation 
		\[
		\ast : \Big(\prod_{i\in I}\mathscr{V}_i\Big)
		\times
		\Big(\prod_{i\in I}\mathscr{V}_i\Big)
		\longrightarrow
		\prod_{i\in I}\mathscr{V}_i,
		\qquad
		(u \ast v)_i := u_i \ast_i v_i,
		\]
		for all $u=(u_i)_{i\in I}$, $v=(v_i)_{i\in I}$ and $i\in I$.
		
		Then $\big(\prod_{i\in I}\mathscr{V}_i,\preceq,\ast\big)$ 
		is again a continuous quantale since the product of continuous complete lattices is again a continuous complete lattice (see \cite[Proposition I-2.1]{GierzHoffKeiLawMisScott}). 
		
		If all members of the family coincide with a fixed continuous quantale 
		$(\mathscr{V},\preceq,\ast)$, we shall use the same symbols 
		$\preceq$ and $\ast$ for the induced order and operation on 
		$\mathscr{V}^I$.
		
		In particular, we denote by $\mathsf{P}_+^{\,I}$ the continuous quantale
		\[
		\mathsf{P}_+^{\,I}
		:=
		\Big([0,+\infty]^I,\, \leq^{\mathrm{op}},\, +\Big),
		\]
		where the order is defined componentwise by the opposite of the usual order
		on $[0,+\infty]$, and the operation $+$ is given componentwise, that is,
		\[
		(u+v)_i = u_i + v_i \qquad \text{ for all }i\in I,
		\]
		with $+$ on the right-hand side denoting the usual addition of extended
		nonnegative real numbers. The unit of $\mathsf{P}_+^{\,I}$ will be denoted by $0_I.$
		
	\end{example}

	As previously observed, in \cite{Flagg97}, Flagg introduced the notion of value quantale by abstracting the order and algebraic properties of
	$\big([0,+\infty],\leq^{\mathrm{op}},+\big)$. Concretely, a value quantale is a quantale $(\mathscr V,\preceq,\ast)$ such that $(\mathscr V,\preceq)$ is completely distributive and the set
	\[ \{x \in \mathscr V : x \lhd \top_\V\}  \]
	is directed, that is, whenever $u \lhd \top_\V$ and $v \lhd \top_\V$, one has     
	\[ u \vee v \lhd \top_\V. \]
	Here, $\lhd$ denotes the well-below relation: \(x\lhd y\) if, for every subset $A\subseteq\mathscr V$ such that 
	\[y\preceq\bigvee A,\]
	there exists $a\in A$ satisfying $x\preceq a.$ Recall that Raney \cite{Raney53} proved that complete distributivity of a lattice $L$ is equivalent to the well-below relation being approximating, that is,
	\[x=\bigvee\{y\in L:y\lhd x\}\]
	for every $x\in L.$
	
	For the specific quantale $\big([0,+\infty],\leq^{\mathrm{op}},+\big)$, the well-below and way-below relations agree except at the bottom element $+\infty$. More precisely,
	\[u\lhd^{\mathrm{op}}v\quad\Longleftrightarrow\quad u>v,\]
	whereas
	\[u\ll^{\mathrm{op}}v\quad\Longleftrightarrow\quad u>v
	\text{ or }u=v=+\infty.\]

	However, a fundamental difference appears as soon as one considers
	Cartesian products.
	Although $\big([0,+\infty],\leq^{\mathrm{op}},+\big)$ is a value quantale,
	its Cartesian product $\mathsf P_+^{\,I}$
	is no longer a value quantale if $I$ contains at least two elements. 
	Indeed, the well-below relation in the product becomes strictly finer: if $a \lhd^{\mathrm{op}} b$ in $[0,+\infty]^I$, then $a$ must be equal to $+\infty$ in all but at most one coordinate.
	
	Assume that $I$ contains at least two distinct elements $i,j\in I$, and define $u,v\in[0,+\infty]^I$ by
	\[u_k=\begin{cases}
		1& k=i,\\
		+\infty& k\neq i,
	\end{cases}
	\qquad
	v_k=\begin{cases}
		1& k=j,\\
		+\infty& k\neq j.
	\end{cases}\]
	Then
	\[u\lhd^{\mathrm{op}}0_I \qquad\text{and}\qquad v\lhd^{\mathrm{op}}0_I.\]
	However, since the join with respect to the order
	$\leq^{\mathrm{op}}$ is given coordinatewise by the minimum with
	respect to the usual order, the element $u\vee v$ has two finite
	coordinates, namely
	\[(u\vee v)_i=(u\vee v)_j=1.\]
	Hence
	\[u\vee v\not\!\!\lhd^{\mathrm{op}}0_I.\]
	Therefore the condition
	\[u\lhd^{\mathrm{op}}0_I \text{ and } v\lhd^{\mathrm{op}}0_I \quad\Longrightarrow\quad u\vee v\lhd^{\mathrm{op}}0_I\]
	required in Flagg's definition fails. Thus, if $I$ contains at least two elements, $\mathsf P_+^{\,I}$ is not a value quantale.

	The source of this phenomenon lies in the divergence of the well-below relation and the way-below relation in products. In product lattices, the way-below relation behaves naturally (finite-support behaviour), whereas the well-below relation
	becomes excessively restrictive. Consequently, to develop a theory of $\mathscr{V}$-categories that is stable under Cartesian products, it is more appropriate to work with continuous quantales. The most relevant topology related to the way-below relation is the Scott topology, which we recall next.

	\begin{definition}[\mbox{\cite[Definitions II-1.3]{GierzHoffKeiLawMisScott}}]
		Let $(L,\preceq)$ be a complete lattice. The \emph{Scott topology} on $L$ is the family of all subsets $O$ of $L$ such that: \begin{itemize}
			\item if $x\succeq y$ and $y\in O$ then $x\in O;$
			\item $\bigvee D\in O$ implies $D\cap O\neq\varnothing,$ for every directed set $D.$
		\end{itemize}
	\end{definition}
	
	Notice that if $(L,\preceq)$ is a continuous lattice, then the family $\{\twoheaduparrow x:x\in L\}$ is a basis for the Scott topology \cite[Proposition II-1.10]{GierzHoffKeiLawMisScott}.

	\begin{example}\label{ex:left_order_topology}
		Consider the complete lattice $([0,+\infty],\leq^{\mathrm{op}})$. In this case, the Scott-open subsets are precisely the lower sets with respect to the usual order which are inaccessible by infima of filtered sets. It follows that the family
		\[
		\big\{[0,a): a\in (0,+\infty]\big\}\cup\big\{[0,+\infty]\big\}
		\]
		forms a basis for the Scott topology.  Hence, it coincides with the left-order topology on $[0,+\infty]$. 
	\end{example}
	
	Combining Example \ref{ex:waybelow}(3) with the previous example, the Scott topology on $\mathsf P_+^{\,I}$ coincides with the product of the left-order topologies on $[0,+\infty]$. In particular, a neighbourhood basis of $0_I$ is formed by the sets
	\[\left\{u\in[0,+\infty]^I: u_j<\varepsilon_j\text{ for every }j\in J\right\}, \]
	where $J\subseteq I$ is finite and $\varepsilon_j>0$ for every $j\in J$.

	\begin{definition}
		A function $f:(L,\preceq_L)\to (S,\preceq_S)$ between two complete lattices is said to be \emph{Scott-continuous} at $x\in L$ if it is continuous at $x$ with respect to the Scott topologies of $L$ and $S.$
	\end{definition}
	
	The theory of enriched categories originated in the foundational works of Bénabou, Eilenberg and Kelly \cite{Benavou63,Benavou65,EilenbergKelly66}, and was later systematically developed by Kelly. A decisive step in its development was the seminal paper \cite{Lawv73}, where Lawvere observed that extended quasi-pseudometric spaces can be understood as categories enriched over the continuous quantale $\mathsf P_+=\big([0,+\infty],\le^{\text{op}},+\big).$
	
	Enrichment over a quantale is a particular instance of enrichment over a symmetric monoidal closed category. Despite this apparent restriction, the quantale-enriched framework captures a broad spectrum of classical structures, including preordered sets and several kinds of generalized metric spaces \cite{Flagg97,ClementinoHofmannTholen04}.
	
	For the purposes of this work, this setting provides the appropriate level of generality. We therefore recall the notion of a category enriched over a quantale.
	
	\begin{definition}[\mbox{\cite[Section III.1.3]{BookMonoidalTopology}, cf. \cite[Definition 3.1]{FlaggKopp97}}]
		Let $(\mathscr{V},\preceq,\ast)$ be a quantale. A \emph{$\mathscr{V}$-category} is a pair $(X,a)$ where $X$ is a nonempty set and $a:X\times X\to\mathscr{V}$ is a map such that
		\begin{itemize}
			\item[(VC1)] $1_\V\preceq a(x,x)$ for all $x\in X;$
			\item[(VC2)] $a(x,z)\ast a(z,y)\preceq a(x,y)$ for all $x,y,z\in X.$
		\end{itemize}
		Moreover, a $\mathscr{V}$-category $(X,a)$ is said to be:
		\begin{itemize}
			\item \emph{separated} if given $x,y \in X$, whenever $1_\V\preceq a(x,y)$ and $1_\V\preceq a(y,x)$ then $x=y.$
			\item \emph{symmetric} if $a(x,y)=a(y,x)$ for all $x,y\in X.$
		\end{itemize}
		
		A function $f:(X,a)\to (Y,b)$ between two $\mathscr{V}$-categories is called a \emph{$\mathscr{V}$-functor} if
		$$a(x,y)\preceq b(f(x),f(y))$$
		for all $x,y\in X.$
		
		The category whose objects are $\mathscr{V}$-categories and whose morphisms are $\mathscr{V}$-functors will be denoted by $\mathscr{V}$-$\mathsf{Cat}.$
		
		If $(\mathscr{V},\preceq,\ast)$ is a continuous quantale, the category $\mathscr{V}$-$\mathsf{Cat}$ will be denoted by $\mathscr{V}$-$\mathsf{Met}$ and its objects will be called \textbf{$\mathscr{V}$-metric spaces} \cite{DagninoFarjudianMoggi23}.
		
	\end{definition}
	
	When $\mathscr{V}$ is a value quantale, $\mathscr{V}$-categories were called \emph{continuity spaces} in \cite{Flagg97} and \emph{Flagg continuity spaces} in \cite{CookWeiss21}.

	\begin{example}[\mbox{\cite[Examples III.1.3.1]{BookMonoidalTopology},\cite{HofReis13}}]\label{ex:vcat}\
		
		\begin{enumerate}[(1)]
			
			\item $\mathbf{2}$-metric spaces correspond to preordered sets. 
			Indeed, a $\mathbf{2}$-category $(X,a)$ induces the preorder 
			$x\preceq_a y \Leftrightarrow a(x,y)=\top$, and conversely a preordered set
			$(X,\preceq)$ defines a $\mathbf{2}$-category via
			\[
			a_\preceq(x,y)=
			\begin{cases}
				\top & x\preceq y,\\
				\bot & \text{otherwise},
			\end{cases}
			\]
			for all $x,y \in X.$ It is also clear that separated
			$\mathbf{2}$-categories correspond to partially ordered sets.
			
			\item For the Lawvere quantale $\mathsf{P}_+$, categories enriched over it are
			precisely extended quasi-pseudometric spaces. Indeed, the axioms read
			\[
			a(x,x)=0, \qquad a(x,y)\le a(x,z)+a(z,y),
			\]
			so $a$ is an extended quasi-pseudometric, and conversely every such
			space defines a $\mathsf{P}_+$-category. Symmetric $\mathsf P_+$-categories correspond to extended pseudometric spaces, separated $\mathsf P_+$-categories correspond to extended quasi-metric spaces, and separated and symmetric $\mathsf P_+$-categories
			correspond to extended metric spaces.
			
			Moreover, a $\mathsf{P}_+$-functor $f:(X,a)\to(Y,b)$ satisfies
			\[
			a(x,y)\ge b(f(x),f(y)),
			\]
			that is, $f$ is non-expansive.
			\item Every quantale $\mathscr{V}$ becomes a $\mathscr{V}$-category over itself by means of the right implication 
			$h : \mathscr{V} \times \mathscr{V} \to \mathscr{V}$ defined by
			\[h(u,v)\;=\;u\searrow v\;=\;\bigvee \bigl\{\, w\in\mathscr{V}:u\ast w\preceq v\,\bigr\},\]
			for all $u,v \in \mathscr{V}$ 
			(see Remark \ref{rem:right_adjoint}, \cite[Proposition~III.1.3.2]{BookMonoidalTopology}). Notice that $h(u,v)$ is precisely the canonical hom-structure of the quantale $\mathscr V$ viewed as a $\mathscr V$-category over itself.
			
		\end{enumerate}
	\end{example}

	\section{Topologically preserving functions between quantales}\label{sec:top_preserving}
	
	In \cite{FBRL25a}, aggregation processes are placed within a unified categorical framework by interpreting some classical structures (such as metric, quasi-metric, and fuzzy quasi-metric spaces) as categories enriched over a quantale. Within this setting, aggregation functions are identified with preserving maps between quantales as we next define.
	
	\begin{definition}[\cite{FBRL25a}]\label{def:catpreserving}
		A map $F:(\mathscr{V},\preceq_\V,\ast)\to(\mathscr{W},\preceq_\W,\star)$ between two quantales is said to be \textbf{$\mathsf{Cat}$-preserving} if the map $\mathsf{F}:\mathscr{V}\text{-}\mathsf{Cat}\to\mathscr{W}\text{-}\mathsf{Cat},$ which assigns to a $\mathscr{V}$-category $(X,a)$ the pair $(X,F\circ a),$ and to a $\mathscr{V}$-functor $f:(X,a)\to (Y,b)$ the map $\mathsf{F}(f):(X,F\circ a)\to (Y,F\circ b)$ given by $\mathsf{F}(f)(x)=f(x)$ for all $x\in X,$ is a functor.
		
		If the functor $\mathsf F$ restricts to a functor from separated $\mathscr V$-categories to separated $\mathscr W$-categories, then $F$ is said to be separately $\mathsf{Cat}$-preserving. Similarly, if the functor $\mathsf F$ restricts to a functor from symmetric $\mathscr V$-categories to symmetric $\mathscr W$-categories, then $F$ is said to be symmetrically $\mathsf{Cat}$-preserving.
	\end{definition}
	
	\begin{definition}[\cite{FBRL25a}]\label{def:preserving}
		A map $F:(\mathscr{V},\preceq_\V,\ast)\to(\mathscr{W},\preceq_\W,\star)$ between two quantales is said to be \textbf{preserving} if the map $\mathsf{F}:\mathrm{Obj}(\mathscr{V}\text{-}\mathsf{Cat})\to\mathrm{Obj}(\mathscr{W}\text{-}\mathsf{Cat}),$ which assigns to a $\mathscr{V}$-category $(X,a)$ the pair $(X,F\circ a)$ is well-defined, that is, if $(X,F\circ a)$ is a $\mathscr{W}$-category.
		
		We will denote by $\mathscr{P}\big((\mathscr{V},\preceq_\V,\ast),(\mathscr{W},\preceq_\W,\star)\big)$, or simply by $\mathscr{P}(\mathscr{V},\mathscr{W})$ if no confusion arises, the family of preserving functions between the quantales $(\mathscr{V},\preceq_\V,\ast)$ and $(\mathscr{W},\preceq_\W,\star).$
		
		If $(X,F\circ a)$ is a separated (resp. symmetric) $\mathscr{W}$-category whenever $(X,a)$ is a separated (resp. symmetric) $\mathscr{V}$-category then $F$ is said to be \textbf{separately preserving} (resp. \textbf{symmetrically preserving}).  The family of separately (resp. symmetrically) preserving functions will be denoted by $\mathsf{se}\mathscr{P}\big((\mathscr{V},\preceq_\V,\ast),(\mathscr{W},\preceq_\W,\star)\big)$ (resp. $\mathsf{sy}\mathscr{P}\big((\mathscr{V},\preceq_\V,\ast),(\mathscr{W},\preceq_\W,\star)\big)$).
	\end{definition}
	
	In \cite{FBRL25a}, it was shown that preserving functions coincide with lax morphisms of quantales. We next recall this concept.

	\begin{definition}[\mbox{\cite[Section II.1.10]{BookMonoidalTopology},\cite{HofReis13}}]
		A map $F:(\mathscr{V},\preceq_\V,\ast)\to(\mathscr{W},\preceq_\W,\star)$ between two quantales is said to be a \emph{lax morphism of quantales} if
		\begin{itemize}
			\item $u\preceq_\V v$ implies $F(u)\preceq_\W F(v)$ for all $u,v\in\mathscr{V};$\hfill (isotone)
			\item $1_\mathscr{W}\preceq_\W F(1_\mathscr{V});$
			\item $F(u)\star F(v)\preceq_\W F(u\ast v)$ for all $u,v\in\mathscr{V}.$
		\end{itemize}
		
	\end{definition}
	
	For suitable choices of the quantales $\mathscr V$ and $\mathscr W$, the characterization of preserving functions as lax morphisms of quantales yields characterizations of aggregation functions for several classes of structures, including quasi-pseudometrics, fuzzy quasi-pseudometrics, quasi-pseudometric modulars, and preorders \cite{FBRL25a,FBRL25b,FBRL27b,MayorVale10,MinyaVale19,BFVale24}.
	
	For example, when $\mathscr V=\mathsf P_+^{\,I}$ and $\mathscr W=\mathsf P_+,$ preserving functions coincide with extended quasi-pseudometric aggregation functions on products and on sets \cite[Lemma 3.7]{FBRL25a}, which is a broader family than the metric-preserving functions studied by Borsík and Doboš \cite{BorsikDobos81b,Dobos98}, that is, mappings $f \colon [0,+\infty)^I \to [0,+\infty)$ such that, for every family of metric spaces $\big\{(X_i,d_i):i\in I\}$, the function obtained by composing $f$ with the product distance defines a metric on $\prod_{i\in I} X_i$.

	Borsík and Doboš also introduced a topological perspective into the aggregation problem through the notion of a strongly metric-preserving function \cite{BorsikDobos81b,Dobos98}, requiring that the product topology induced by the metric topologies of the factors coincide with the topology induced by the aggregated metric on $\prod_{i\in I}X_i$.
	
	Our aim is to extend this idea to the broader context of preserving functions defined above. More precisely, we investigate to what extent a preserving map $F \colon \mathscr{V} \to \mathscr{W}$ ensures not only the preservation of the enriched categorical structure, but also the induced topologies. In this way, we develop a generalized notion of strongly preserving function which encompasses the classical strongly metric-preserving functions as a particular case and applies to other enriched settings.

	Given a continuous quantale $(\mathscr{V},\preceq,\ast),$ there is a natural topology underlying a $\mathscr{V}$-metric space which can be constructed in a similar way to the metric topology. This theory was developed by Flagg and Kopperman \cite{Flagg97,FlaggKopp97} and more recently by Cook and Weiss \cite{CookWeiss21}, when $\mathscr{V}$ is a value quantale. For continuous quantales it has been studied in \cite{DagninoFarjudianMoggi23}.
	
	\begin{definition}[\mbox{\cite{DagninoFarjudianMoggi23}, cf. \cite{Flagg97,CookWeiss21}}]
		Let $(\mathscr{V},\preceq,\ast)$ be a continuous quantale and let $(X,a)$ be a $\mathscr{V}$-metric space. The family
		$$\mathscr{T}(a)=\{G\subseteq X: \text{for every }x\in G\text{ there exists }\varepsilon\ll 1_\V\text{ such that }B_a(x,\varepsilon)\subseteq G\}$$
		where 
		$$B_a(x,\varepsilon)=\{y\in X: \varepsilon \ll a(x,y)\},$$
		is a topology on $X.$ Moreover $\{B_a(x,\varepsilon):x\in X,\varepsilon\ll 1_\V\}$ is a base for the topology $\mathscr{T}(a).$
	\end{definition}

	\begin{example}\label{ex:product}\
		\begin{enumerate}[(1)]
			\item If $(X,a)$ is a $\mathbf{2}$-metric space then the topology $\mathscr{T}(a)$ is the Alexandroff topology of the preordered set $(X,\preceq_a)$ where $\preceq_a$ is the preorder defined by $x\preceq_a y$ if and only if $a(x,y)=\top$ (see \cite[Example 4.9]{Flagg97}).
			
			\item If $(X,a)$ is a $\mathsf{P}_+$-metric space then the topology $\mathscr{T}(a)$ is the standard topology induced by the extended quasi-pseudometric space $a$ on $X$ \cite[Example 4.10]{Flagg97},\cite{KunziCh09}.
			
			\item Let $\mathsf{P}_+^{\,I}=\Big([0,+\infty]^I,\, \leq^{\mathrm{op}},\, +\Big)$ be the continuous quantale considered in Example \ref{ex:product_quantales} and let $(X,a)$ be a $\mathsf{P}_+^{\,I}$-metric space. Given $\varepsilon\in [0,+\infty]^I$ then $\varepsilon \ll^{\mathrm{op}} 0_I$ if and only if $\varepsilon_i<^{\mathrm{op}} 0$ and $\varepsilon_i=\infty$ for all but finitely many $i$ (see Example \ref{ex:waybelow}). Therefore, the set $J=\{j\in I:\varepsilon_j\neq\infty\}$ is finite. Hence, if $x\in X$ then
			\begin{align*}
				B_{a}(x,\varepsilon)&=\left\{y\in X: \varepsilon\ll^{\mathrm{op}}a(x,y)\right\}\\
				&=\left\{y\in X: (\pi_j\circ a)(x,y)<\varepsilon_j \text{ for all } j\in J\right\}\\
				&=\bigcap_{j\in J}\left\{y\in X: (\pi_j\circ a)(x,y)<\varepsilon_j \right\}\\
				&=\bigcap_{j\in J}\left\{y\in X: \varepsilon_j\ll^{\mathrm{op}} (\pi_j\circ a)(x,y) \right\}\\
				&=\bigcap_{j\in J} B_{\pi_j\circ a} (x,\varepsilon_j)
			\end{align*}
			where $\pi_j$ denotes the projection onto the $j$-th coordinate. In addition, it is clear that $(X,\pi_i\circ a)$ is a $\mathsf{P}_+$-metric space for every $i\in I,$ that is, an extended quasi-pseudometric space.   
			Hence $\mathscr{T}(a)$ is precisely the supremum topology $\bigvee \mathscr{T}(\pi_i\circ a)$ of the family of topologies $\{\mathscr{T}(\pi_i\circ a):i\in I\}$ on $X$.
			
			\item We next apply the previous characterization of the topology on a $\mathsf P_+^{\,I}$-metric space to one particular case. Let $\big\{(X_i,d_i):i\in I\big\}$ be a family of $\mathsf{P}_+$-metric spaces. Then $(\prod_{i\in I}X_i,d_\Pi)$ is a $\mathsf{P}_+^{\,I}$-metric space. Given $x\in \prod_{i\in I} X_i$ and $\varepsilon\ll^{\mathrm{op}} 0_I$, let $J$ be the finite set $\{j\in I:\varepsilon_j\neq\infty\}$. Then
			\begin{align*}
				B_{d_\Pi}(x,\varepsilon)&=\left\{y\in\prod_{i\in I} X_i: \varepsilon\ll^{\mathrm{op}}d_\Pi(x,y)\right\}\\
				&=\left\{y\in\prod_{i\in I}X_i: d_j(x_j,y_j)<\varepsilon_j \text{ for all } j\in J\right\}.
			\end{align*}
			
			Therefore, the topology $\mathscr{T}(d_\Pi)$ coincides with the product topology on $\prod_{i\in I} X_i$ that we will denote by $\prod\mathscr{T}(d_i).$
		\end{enumerate}
	\end{example}
	
	If we denote by $\mathsf{Top}$ the category of topological spaces and continuous maps, then we can define a functor
	\[ \mathsf{T}_{\mathscr V}:\mathscr{V}\text{-}\mathsf{Met}\to\mathsf{Top}\]
	sending each $\mathscr{V}$-metric space $(X,a)$ to the topological space $(X,\mathscr{T}(a))$ and acting as the identity on morphisms (see \cite{CookWeiss21}). Notice that $\mathsf{T}_{\mathscr V}$ is well defined since, for every morphism $f:(X,a)\to (Y,b)$
	between $\mathscr{V}$-metric spaces, every $x\in X,$ and every $\varepsilon\ll 1_{\mathscr V},$ one has
	\[f\big(B_a(x,\varepsilon)\big)\subseteq B_b\big(f(x),\varepsilon\big).\]
	Hence,
	\[f:\big(X,\mathscr{T}(a)\big)\to \big(Y,\mathscr{T}(b)\big)\]
	is continuous (see \cite[Theorem 1]{DagninoFarjudianMoggi23}).
	
	We now extend the notion of a strongly metric-preserving function to the setting of $\mathscr V$-metric spaces.
	
	\begin{definition}
		A function $F:(\mathscr{V},\preceq_\V,\ast)\to(\mathscr{W},\preceq_\W,\star)$ between two continuous quantales is said to be \textbf{strongly $\mathsf{Cat}$-preserving} if $F$ induces a functor 
		\[\mathsf{F}:\mathscr V\text{-}\mathsf{Met}\to \mathscr W\text{-}\mathsf{Met},\]
		that is, it is $\mathsf{Cat}$-preserving, and the following diagram commutes strictly
		\begin{center}
			
			\begin{tikzpicture}[scale=1.5]
				\node (A) at (0,0) {$\mathscr{V}\text{-}\mathsf{Met}$};
				\node (B) at (2,0) {$\mathscr{W}\text{-}\mathsf{Met}$};
				\node (C) at (2,-1) {$\mathsf{Top}$};
				\node (AA) at (0.2,0) {\phantom{$\mathsf{Unif}$}};
				\node (BB) at (1.85,0.0) {\phantom{$\mathsf{SUnif}$}};
				\node (CC) at (2,-1) {\phantom{$\mathsf{LUnif}$}};
				\path[->,font=\scriptsize] (AA) edge node[above]{$\mathsf{F}$} (BB);
				\path[->,font=\scriptsize] (B) edge node[right]{$\mathsf{T}_\W$} (C);
				\path[->,font=\scriptsize] (A) edge node[below]{$\mathsf{T}_\V$} (C);
			\end{tikzpicture}
		\end{center}
		that is,
		\[\mathsf{T}_{\mathscr W}\circ \mathsf{F} = \mathsf{T}_{\mathscr V}.\]
		
		If, for every $\mathscr V$-metric space $(X,a)$, the pair
		$(X,F\circ a)$ is a $\mathscr W$-metric space and
		\[\mathscr T(F\circ a)=\mathscr T(a),\]
		then $F$ is said to be \textbf{strongly preserving}.
		
		Furthermore, $F$ is said to be \textbf{upper strongly preserving} if, for every $\mathscr V$-metric space $(X,a),$ $(X,F\circ a)$ is a $\mathscr W$-metric space and 
		\[
		\mathscr T(F\circ a)\subseteq \mathscr T(a),
		\]
		and \textbf{lower strongly preserving} if, for every $\mathscr V$-metric
		space $(X,a)$, $(X,F\circ a)$ is a $\mathscr W$-metric space and
		\[
		\mathscr T(a)\subseteq \mathscr T(F\circ a).
		\]
		
		The previous notions admit analogous versions when considering only
		separated, symmetric, or separated and symmetric $\mathscr V$-metric
		spaces.
		
	\end{definition}
	
	It was proved in \cite[Theorems~3.13 and~3.17]{FBRL25a} that a function
	is (separately) $\mathsf{Cat}$-preserving if and only if it is
	(separately) preserving (the corresponding equivalence does not hold in
	the symmetric case \cite[Example~3.27]{FBRL25a}). Since the functors
	$\mathsf T_{\mathscr V}$ and $\mathsf T_{\mathscr W}$ act as the identity
	on the underlying maps of morphisms, we obtain the following result.
	
	\begin{proposition}
		Let $F:(\mathscr{V},\preceq_\V,\ast)\to(\mathscr{W},\preceq_\W,\star)$ be a function between two continuous quantales. The following statements are equivalent:
		\begin{enumerate}[(1)]
			\item $F$ is strongly (separately) $\mathsf{Cat}$-preserving;
			\item $F$ is strongly (separately) preserving.
		\end{enumerate}
	\end{proposition}
	
	\subsection{Upper strongly preserving functions}

	This section is devoted to the characterization of upper strongly preserving functions.
	
	A natural first question is whether every preserving function is automatically upper strongly preserving. The following simple example shows that this is not the case.
	
	\begin{example}
		Let
		\[
		\mathscr V=\mathbb N\cup\{\infty\}
		\]
		be endowed with the usual order $\leq.$ Then $(\mathscr{V},\leq,\wedge)$ is a continuous quantale. 
		Define $F:\mathscr V\to \mathbf{2}$ by
		\[F(\infty)=\top, \qquad F(n)=\bot \quad\text{for every }n\in\mathbb N. \]
		It is easy to check that $F$ is a lax morphism of quantales. Hence $F$ is preserving by \cite[Theorem 3.13]{FBRL25a}.
		
		However, $F$ is not upper strongly preserving. To see this, define $a:\mathscr V\times \mathscr V\to\mathscr V$ by
		\[a(n,m)=\begin{cases}
			\infty&\text{if }n=m,\\
			n\wedge m&\text{if }n\neq m,
		\end{cases}\]
		for every $n,m\in\mathscr V.$ Then $(\mathscr V,a)$ is a $\mathscr V$-metric space. Furthermore,  the topology $\mathscr T(a)$ is not discrete. In fact, for every $n\in\mathbb N$,
		\[B_a(\infty,n)=\{\infty\}\cup\{m\in\mathbb N:n\leq m\}.\]
		Thus, every basic neighbourhood of $\infty$ contains infinitely many natural numbers, and hence $\{\infty\}$ is not $\mathscr T(a)$-open.
		
		On the other hand, $F\circ a$ is the $\mathbf 2$-metric given by
		\[(F\circ a)(n,m)=\begin{cases}
			\top&\text{if }n=m,\\
			\bot&\text{if }n\neq m.
		\end{cases}\]
		Since $\top\ll_\mathscr W\top$, we have
		\[B_{F\circ a}(n,\top)=\{n\}\]
		for every $n\in \mathscr V$. Therefore $\mathscr T(F\circ a)$ is the discrete topology.
		Consequently,
		\[\mathscr T(F\circ a)\nsubseteq \mathscr T(a),\]
		and so $F$ is not upper strongly preserving. In particular, $F$ is not strongly preserving.
	\end{example}
	
	Notice that the preserving function $F$ of the previous example is not Scott continuous at $1_\V=\infty.$ In fact $\twoheaduparrow F(\infty)=\twoheaduparrow \top =\{\top\}$ is a Scott neighborhood of $F(\infty).$  On the other hand, since
	$\infty\not\ll_{\mathscr V}\infty$, every basic Scott neighborhood of $\infty$ is of the form $\twoheaduparrow n$ for some $n\in\mathbb N$. Hence every Scott open neighborhood $U$ of $\infty$ contains infinitely many natural
	numbers, and therefore $F(U)=\{\bot,\top\}\nsubseteq \{\top\}.$ Thus $F$ is not Scott continuous at $1_{\mathscr V}$. We next show that this fact is not accidental.

	\begin{theorem}\label{thm:us}
		Let $F:(\mathscr{V},\preceq_\V,\ast)\to(\mathscr{W},\preceq_\W,\star)$ be a preserving function between continuous quantales. The following statements are equivalent:
		\begin{itemize}
			\item[(1)] $F$ is upper strongly preserving$;$
			\item[(2)] $F$ is Scott continuous at $1_\V.$
			
		\end{itemize}

		The previous equivalence also remains valid for symmetrically upper strongly preserving and separately upper strongly preserving functions.\medskip
		
		Moreover, in those preservation settings in which the functions under
		consideration are isotone (for instance, in the preserving, separately preserving, and symmetrically $\mathsf{Cat}$-preserving settings), the previous equivalent conditions are also
		equivalent to:
		\begin{itemize}
			\item[(3)] $F(1_\V)=\bigvee \bigl\{F(u):u\ll_\V 1_\V\bigr\}.$
		\end{itemize}
	\end{theorem}
	
	\begin{proof}
		
		(1) $\Rightarrow$ (2). Let us consider the separated and symmetric $\mathscr{V}$-metric space $
		(\mathscr V,a)$ associated with the symmetrization of the residuation $h$ on $\mathscr{V}$ (Example \ref{ex:vcat}), namely,
		$$a(u,v)=h(u,v)\ast h(v,u)=(u\searrow v)\ast(v\searrow u)\qquad \text{ for every }u,v\in\mathscr{V}.$$
		Given $\varepsilon \ll_\W 1_\W$, then $\twoheaduparrow\varepsilon$ is a basic open set in the Scott topology containing $1_\W.$ 
		By assumption, there exists $\delta \ll_\V 1_\V$ such that
		\[
		B_a(1_\V,\delta) 
		\subseteq 
		B_{F \circ a}(1_\V,\varepsilon).
		\]
		Let $\gamma \in \twoheaduparrow\delta$. 
		Since
		\[a(1_\mathscr V,\gamma)=h(1_\mathscr V,\gamma)\ast h(\gamma,1_\mathscr V)=\gamma,\]
		it follows that
		\[
		\gamma \in B_a(1_\V,\delta)
		\subseteq 
		B_{F \circ a}(1_\V,\varepsilon).
		\]
		Hence,
		\[
		\varepsilon \ll_\W F \circ a(1_\V,\gamma) 
		= F(\gamma),
		\]
		so $F(\gamma)\in\twoheaduparrow\varepsilon.$ Therefore $F$ is Scott continuous at $1_\V$.
		
		(2) $\Rightarrow$ (1). Let $(X,a)$ be a $\mathscr{V}$-metric space in the class under consideration. We prove that 
		\[\mathscr{T}(F \circ a) \subseteq \mathscr{T}(a).\]
		Let $G \in \mathscr{T}(F \circ a)$ and $x \in G$. Then there exists $\varepsilon \ll_\W 1_\W$ such that
		\[B_{F \circ a}(x,\varepsilon) \subseteq G.\]
		Since $F$ is Scott continuous at $1_\mathscr V,$ we can find $\delta\ll 1_\mathscr{V}$ such that
		\[F(\twoheaduparrow \delta)\subseteq\twoheaduparrow \varepsilon.\]
		We claim that 
		\[B_a(x,\delta) \subseteq B_{F \circ a}(x,\varepsilon).\]
		If $y\in B_a(x,\delta)$ then $a(x,y)\in\twoheaduparrow\delta$ so $F(a(x,y))\in\twoheaduparrow \varepsilon,$ that is, $y\in B_{F \circ a}(x,\varepsilon).$ This proves the claim.

		(2)$\Rightarrow$ (3). This implication does not require any isotonicity assumption. Indeed, in any preservation setting,  by \cite{FBRL25a}  $F(1_\mathscr V)=1_\mathscr W$ and since $\mathscr W$ is continuous, we have
		\[F(1_\mathscr V)=1_\mathscr W=\bigvee \twoheaddownarrow 1_\mathscr W.\]
		Let $\varepsilon\ll_\mathscr W 1_\mathscr W.$ Since $F$ is Scott continuous at $1_\mathscr V,$ there exists
		$\delta\ll_\mathscr V 1_\mathscr V$ such that
		\[F(\twoheaduparrow\delta)\subseteq \twoheaduparrow\varepsilon.\]
		By the interpolation property, there exists
		$u_\varepsilon\in\mathscr V$
		such that
		\[\delta\ll_\mathscr V u_\varepsilon\ll_\mathscr V 1_\mathscr V.\]
		Hence $u_\varepsilon\in\twoheaduparrow\delta,$ and therefore $F(u_\varepsilon)\in\twoheaduparrow\varepsilon,$ that is, $\varepsilon\ll_\mathscr W F(u_\varepsilon).$
		Consequently,
		\[1_\mathscr W=\bigvee\{\varepsilon\in\mathscr W:\varepsilon\ll_\mathscr W 1_\mathscr W\}
		\preceq_\mathscr W
		\bigvee\{F(u_\varepsilon):\varepsilon\ll_\mathscr W 1_\mathscr W\}
		\preceq_\mathscr W
		\bigvee\{F(u):u\ll_\mathscr V 1_\mathscr V\}.
		\]
		Since $1_\mathscr W$ is the top element of $\mathscr W,$ it follows that
		\[F(1_\mathscr V)=\bigvee\{F(u):u\ll_\mathscr V 1_\mathscr V\}.\]
		
		(3) $\Rightarrow$ (1). 	
		Assume that we are in any of the preservation settings under consideration in which $F$ is isotone. Moreover, $F(1_\V)=1_\W$ by \cite[Theorems 3.13, 3.17, Corollary 3.26]{FBRL25a}.
		
		Let $(X,a)$ be a $\mathscr{V}$-metric space belonging to the
		preservation setting under consideration. We must prove that $\mathscr{T}(F \circ a) \subseteq \mathscr{T}(a).$
		Given $G \in \mathscr{T}(F \circ a)$ and $x \in G$, we can find $\varepsilon \ll_\W 1_\W$ such that
		\[ B_{F \circ a}(x,\varepsilon) \subseteq G.\]
		Since $F$ is isotone, $\bigl\{ F(u):u \ll_\V 1_\V\bigr\}$ is directed and
		\[
		1_\W 
		= F(1_\V)
		= \bigvee \bigl\{ F(u):u \ll_\V 1_\V \bigr\}.
		\]
		By \cite[Theorem I-1.9]{GierzHoffKeiLawMisScott}, there exists 
		$\delta\ll_\V 1_\V$ such that $\varepsilon\ll_\W F(\delta).$     
		
		We next show that
		\[
		B_a(x,\delta) \subseteq B_{F \circ a}(x,\varepsilon).
		\]
		For every $y \in B_a(x,\delta)$, 
		$\delta \ll_\V a(x,y)$. 
		Since $F$ is isotone, it follows that $F(\delta) \preceq_\W F(a(x,y)).$
		Combining this with $\varepsilon \ll_\W F(\delta)$, we obtain
		$\varepsilon \ll_\W F(a(x,y)),$
		which shows that $y \in B_{F \circ a}(x,\varepsilon)$.
		
		Consequently,
		\[
		B_a(x,\delta) \subseteq B_{F \circ a}(x,\varepsilon) \subseteq G,
		\]
		and therefore $G \in \mathscr{T}(a)$.
	\end{proof}

	To illustrate the applicability of the previous result, we will focus on specific continuous quantales. In this context, let's consider a preserving function $F: \mathsf{P}_+^{\,I} \to \mathsf{P}_+.$ Next, we will interpret what it means for $F$ to be an upper strongly preserving function. To accomplish this, we recall the concept of an extended quasi-pseudometric aggregation function on products and sets, as presented in \cite{MassaVale13,FBRL25a}.
	
	\begin{definition}\label{def:eqpmaf}
		A function $F: [0,+\infty]^I\rightarrow [0,+\infty]$ is said to be 
		\begin{itemize}
			\item an \textbf{extended quasi-pseudometric aggregation function on products} if whenever $\big\{(X_i,d_i):i\in I\big\}$ is a family of extended quasi-pseudometric spaces, then $F\circ d_\Pi$ is an extended quasi-pseudometric on $\prod_{i\in I} X_i$ where $d_\Pi:\left(\prod_{i\in I} X_i\right)\times \left(\prod_{i\in I} X_i\right)\to [0,+\infty]^I$ is given by
			$$d_\Pi(x,y)=\big(d_i(x_i,y_i)\big)_{i\in I}$$
			for every $x,y\in \prod_{i\in I} X_i;$
			\item an \textbf{extended quasi-pseudometric aggregation function on sets} if whenever $\{d_{i}:i \in I\}$ is a family of extended quasi-pseudometrics on a nonempty set $X$, then $F \circ d_\Delta$ is an extended quasi-pseudometric on $X$, where $d_\Delta: X \times X \rightarrow [0,+\infty]^{I}$ is given by $$d_\Delta(x,y)=\big(d_{i}(x,y)\big)_{i \in I}$$
			for every $x,y \in X.$
		\end{itemize}
		
		Similar definitions can be stated for quasi-metrics, pseudometrics and metrics.
	\end{definition}

	\begin{definition}\label{def:upper_s_qpmaf}
		Let $F:[0,+\infty]^I\to [0,+\infty]$ be a map. 
		\begin{enumerate}[(1)]
			\item If $F$ is an extended (quasi-)(pseudo)metric aggregation function on products and, for every family $\big\{(X_i,d_i):i\in I\big\}$ of extended (quasi-)(pseudo)metric spaces, it holds that 
			$$\mathscr{T}(F\circ d_\Pi)\subseteq \prod \mathscr{T}(d_i),$$  
			then we say that $F$ is an \textbf{upper strongly extended (quasi-)(pseudo)metric aggregation function on products}.\\
			When the converse inclusion holds
			$$\prod \mathscr{T}(d_i)\subseteq \mathscr{T}(F\circ d_\Pi),$$
			$F$ is called a \textbf{lower strongly extended (quasi-)(pseudo)metric aggregation function on products}.
			\item If $F$ is an extended (quasi-)(pseudo)metric aggregation function on sets and, for every family $\big\{(X,d_i):i\in I\big\}$ of extended (quasi-)(pseudo)metric spaces, it holds that 
			$$\mathscr{T}(F\circ d_\Delta)\subseteq \bigvee \mathscr{T}(d_i),$$ 
			then we say that $F$ is an \textbf{upper strongly extended (quasi-)(pseudo)metric aggregation function on sets}. \\
			When the converse inclusion holds
			$$\bigvee \mathscr{T}(d_i)\subseteq \mathscr{T}(F\circ d_\Delta),$$
			$F$ is called a \textbf{lower strongly extended (quasi-)(pseudo)metric aggregation function on sets}.
		\end{enumerate}
		
	\end{definition}
	
	We next prove that, for functions $F:\mathsf{P}_+^{\,I}\to \mathsf{P}_+$, upper strongly preserving is equivalent to being an upper strongly extended quasi-pseudometric aggregation function on products or on sets.

	\begin{proposition}\label{prop:upper_preserving_equal_upper_eqpm}
		Let $F:\mathsf{P}_+^{\,I}\to \mathsf{P}_+$ be a preserving function (resp. a symmetrically preserving function). The following statements are equivalent:
		\begin{enumerate}[(1)]
			\item $F$ is upper strongly preserving (resp. upper strongly symmetrically preserving);
			\item $F$ is an upper strongly extended quasi-pseudometric aggregation function on products (resp. an upper strongly extended pseudometric aggregation function on products);
			\item $F$ is an upper strongly extended quasi-pseudometric aggregation function on sets (resp. an upper strongly extended pseudometric aggregation function on sets).
		\end{enumerate}
	\end{proposition}
	
	\begin{proof}
		(1) $\Rightarrow$ (2) Since $F$ is preserving, it follows from \cite[Lemma 3.7]{FBRL25a} (see \cite[Remark 3.9]{FBRL25a} for the symmetric setting) that $F$ is an extended quasi-pseudometric aggregation function on products (resp. an extended pseudometric aggregation function on products). \\Consider an arbitrary family $\big\{(X_i,d_i):i\in I\big\}$ of extended quasi-pseudometric spaces (resp. extended pseudometric spaces). By Example \ref{ex:product}, $(\prod_{i\in I} X_i,d_\Pi)$ is a  $\mathsf{P}_+^{\,I}$-metric space (resp. a symmetric $\mathsf{P}_+^{\,I}$-metric space) so $(\prod_{i\in I} X_i,F\circ d_\Pi)$ is a $\mathsf{P}_+$-metric space (resp. a symmetric $\mathsf{P}_+$-metric space), that is, an extended quasi-pseudometric space (resp. an extended pseudometric space). By assumption $$\mathscr{T}(F\circ d_\Pi)\subseteq \mathscr{T} (d_\Pi)$$
		and since, by Example \ref{ex:product},  $\mathscr{T}(d_\Pi)$ is equal to the product topology, the statement holds.
		
		(2) $\Rightarrow$ (3) Recall that every extended \mbox{quasi-}pseu\-dometric aggregation function on products is an extended quasi-pseudo\-metric aggregation function on sets (in fact, the two concepts coincide \cite[Lemma 3.7, Remark 3.9]{FBRL25a}). This observation, together with the fact that, for every family $\big\{(X,d_i):i\in I\big\}$ of extended quasi-pseudometric spaces then $\big(X,\bigvee_{i\in I}\mathscr{T}(d_i)\big)$ is homeomorphic to the diagonal of $\big(\prod_{i\in I} X_i,\prod\mathscr{T}(d_i)\big)$ (see \cite{FBRL26a}), yields the desired inclusion.
		
		The proof for upper strongly extended pseudometric aggregation functions on products is analogous.
		
		(3) $\Rightarrow$ (1) By \cite[Lemma 3.7, Remark 3.9]{FBRL25a}, $F$ is a preserving function. Furthermore,  if $(X,a)$ is a $\mathsf{P}_+^{\,I}$-metric space then $\big\{(X,d_i):i\in I\big\}$ is a family of extended quasi-pseudometric spaces, where $d_i=\pi_i\circ a$ and $\pi_i$ is the $i$th coordinate projection for every $i\in I$. It is clear that
		$$d_\Delta=a$$
		so $\mathscr{T}(F\circ d_\Delta)=\mathscr{T}(F\circ a)\subseteq \bigvee\mathscr{T}(d_i)=\mathscr{T}(a)$ (see Example \ref{ex:product}(3)).
		
		The proof for upper strongly extended pseudometric aggregation functions on sets is analogous.
	\end{proof}
	
	Since the family $\{[0,\varepsilon):\varepsilon>0\}$ forms a basis for the Scott topology on $([0,+\infty],\leq^{\mathrm{op}})$, Theorem \ref{thm:us}, Example \ref{ex:left_order_topology}, and the previous result immediately yield the following corollary.
	
	\begin{corollary}[\mbox{cf.\cite[Theorems 3.8 and 4.15]{FBRL26a}}]\label{cor:upper_strongly_ext_qpm}
		Let $F:[0,+\infty]^I\to [0,+\infty]$ be an extended quasi-pseudometric aggregation function on products 
		(resp. an extended pseudometric aggregation function on products). The following statements are equivalent:
		\begin{enumerate}[(1)]
			\item $F$ is an upper strongly extended quasi-pseudometric aggregation function on products 
			(resp. an upper strongly extended pseudometric aggregation function on products);
			
			\item $F$ is an upper strongly extended quasi-pseudometric aggregation function on sets 
			(resp. an upper strongly extended pseudometric aggregation function on sets);
			
			\item $F$ is continuous at $0_I$, when $[0,+\infty]$ is endowed with the left-order topology.
		\end{enumerate}
		
		Moreover, in the quasi-pseudometric case, the previous statements are also equivalent to:
		\begin{enumerate}[(4)] 
			\item $F(0_I)=\bigvee \bigl\{F(u):u\ll_{\mathsf{P}_+^I} 0_I\bigr\}.$
		\end{enumerate}
	\end{corollary}

	We note that the Proposition \ref{prop:upper_preserving_equal_upper_eqpm} does not hold for upper strongly extended (quasi-)metric aggregation functions, as illustrated by the following example.
	
	\begin{example}
		Consider the projection $P:[0,+\infty]^2\to [0,+\infty]$  onto the first coordinate. It is straightforward to verify that $P$ is an extended \mbox{(quasi-)}metric aggregation function on sets. Moreover, for every nonempty set $X$ and every pair of extended (quasi-)metrics $d_1,d_2$ on $X$, one has
		\[\mathscr{T}(P\circ d_\Delta)=\mathscr{T}(d_1)\subseteq \mathscr{T}(d_1)\vee \mathscr{T}(d_2).\]
		Hence, $P$ is an upper strongly extended (quasi-)metric aggregation function on sets.
		
		Nevertheless, $P$ is not an extended (quasi-)metric aggregation function on products. Indeed, let $\mathbb{R}$ be endowed with the Euclidean metric $e$. Then
		\[(P\circ e_\Pi)\big((0,1),(0,0)\big)=P\big(e(0,0),e(1,0)\big)
		=0,\]
		even though $(0,1)\neq (0,0)$. Consequently, $P\circ e_\Pi$ fails to satisfy the separation axiom and therefore is not a (quasi-)metric on $\mathbb{R}^2$.
	\end{example}

	However, although the equivalence among the three assertions in Proposition~\ref{prop:upper_preserving_equal_upper_eqpm} cannot be obtained in the upper strongly separately preserving setting, the equivalence between the first two assertions can still be established. Before proving this, we first deal with the case of functions that are separately and symmetrically preserving, which was not explicitly addressed in \cite{FBRL25a}. We will use the following concept:
	
	\begin{definition}
		Let $(\mathscr{V},\preceq_\V,\ast)$ be a quantale. A triple
		$(u,v,w)\in\mathscr{V}^3$ is said to be an \textbf{asymmetric triangle
			triplet} in $(\mathscr{V},\preceq_\V,\ast)$ if
		\[u\ast v\preceq_\V w.\]
		It is said to be a \textbf{triangle triplet} if every permutation of
		$(u,v,w)$ is an asymmetric triangle triplet.
		
		A function
		\[F:(\mathscr{V},\preceq_\V,\ast)\longrightarrow (\mathscr{W},\preceq_\W,\star)\]
		between two quantales is said to \textbf{preserve asymmetric triangle
			triplets} if
		\[(F(u),F(v),F(w))\]
		is an asymmetric triangle triplet in $(\mathscr{W},\preceq_\W,\star)$
		whenever $(u,v,w)$ is an asymmetric triangle triplet in
		$(\mathscr{V},\preceq_\V,\ast)$.
		
		Similarly, $F$ is said to \textbf{preserve triangle triplets} if
		\[(F(u),F(v),F(w))\]
		is a triangle triplet in $(\mathscr{W},\preceq_\W,\star)$ whenever
		$(u,v,w)$ is a triangle triplet in $(\mathscr{V},\preceq_\V,\ast)$.
	\end{definition}
	
	\begin{theorem}\label{thm:sep_sym_preserving}
		Let $F:(\mathscr{V},\preceq_\V,\ast)\to (\mathscr{W},\preceq_\W,\star)$ be a map between two commutative integral quantales. Then $F$ is separately and symmetrically preserving if and only if $F$ preserves triangle triplets and $F^{-1}(1_\W)=\{1_\V\}.$
	\end{theorem}
	
	\begin{proof} Suppose first that $F$ is separately and symmetrically preserving. 
		
		First, considering the one-point $\mathscr V$-category, we obtain
		$F(1_\V)=1_\W$. Suppose now that $u\neq1_\V$ and
		$F(u)=1_\W$. Let $X=\{x,y\}$ and define $a:X\times X\to\mathscr{V}$ as
		\[a(x,x)=a(y,y)=1_\V,\qquad a(x,y)=a(y,x)=u.\]
		Since $u\ast u\preceq_\V u$, the pair $(X,a)$ is a separated and
		symmetric $\mathscr V$-category. However, $(X,F\circ a)$ is not
		separated, because
		\[F(a(x,y))=F(a(y,x))=1_\W,\]
		a contradiction. Hence
		\[F^{-1}(1_\W)=\{1_\V\}.\]
		
		We next prove that $F$ preserves triangle triplets. Let $(u,v,w)$ be a triangle triplet in $\mathscr V$. If one of its entries is $1_\V$, say $w=1_\V$, then the triangle inequalities imply $u=v$. Hence the triplet is of the form $(u,u,1_\V)$, and its image 
		\[ (F(u),F(u),F(1_\V))=(F(u),F(u),1_\W) \] 
		is clearly a triangle triplet in $\mathscr W$. We may therefore assume that none of $u,v,w$ is equal to $1_\V$. Let $X=\{x_1,x_2,x_3\}$ and define $a:X\times X\to\mathscr V$ by 
		\[ a(x_i,x_j)= \begin{cases} 
			1_\V &\text{if } i=j,\\ 
			u &\text{if } \{i,j\}=\{1,2\},\\ 
			v &\text{if } \{i,j\}=\{2,3\},\\ 
			w &\text{if } \{i,j\}=\{1,3\}. \end{cases} 
		\] 
		Since $(u,v,w)$ is a triangle triplet, $(X,a)$ is a separated and symmetric $\mathscr V$-category. By hypothesis, $(X,F\circ a)$ is a separated and symmetric $\mathscr W$-category. In particular, applying the triangle inequality to the points $x_1,x_2,x_3$, we obtain 
		\[ F(u)\star F(v) = F(a(x_1,x_2))\star F(a(x_2,x_3)) \preceq_\W F(a(x_1,x_3)) = F(w). \] 
		The other two triangle inequalities are obtained in the same way by permuting the points $x_1,x_2,x_3$. Hence $(F(u),F(v),F(w))$ is a triangle triplet in $\mathscr W$. 
		
		Conversely, by \cite[Theorem 3.21]{FBRL25a}, $F$ is symmetrically preserving. Hence, if $(X,a)$ is a separated and symmetric $\mathscr V$-category, then $(X,F\circ a)$ is a symmetric $\mathscr W$-category. It remains only to prove that it is separated.
		
		Let $x,y\in X$ be such that $F(a(x,y))=1_\W.$ By the assumption on the fibre of $1_\W,$ we get $a(x,y)=1_\V.$ Since $(X,a)$ is symmetric and separated, it follows that $x=y.$ Therefore $(X,F\circ a)$ is separated, and so $F$ is separately and symmetrically preserving. 
	\end{proof}

	\begin{proposition}\label{prop:upper_seppreserving_equal_upper_eqm}
		Let $F:\mathsf{P}_+^{\,I}\to \mathsf{P}_+$. The following statements are equivalent:
		\begin{enumerate}[(1)]
			\item $F$ is upper strongly separately preserving (upper strongly separately and symmetrically preserving);
			\item $F$ is an upper strongly extended quasi-metric aggregation function on products (upper strongly extended metric aggregation function on products).        
		\end{enumerate}
	\end{proposition}
	
	\begin{proof}
		
		(1) $\Rightarrow$ (2) The proof is analogous to the corresponding implication in Proposition \ref{prop:upper_preserving_equal_upper_eqpm}. For the separately preserving case, we apply \cite[Theorem 3.18]{FBRL25a}, whereas for the separated and symmetric setting we use Theorem \ref{thm:sep_sym_preserving}. These results ensure that $F$ is, respectively, an extended quasi-metric aggregation function on products or an extended metric aggregation function on products.  The required inclusion of topologies follows exactly as in Proposition \ref{prop:upper_preserving_equal_upper_eqpm}.
		
		(2) $\Rightarrow$ (1)   In the quasi-metric case, by \cite[Theorem 3.18]{FBRL25a}, $F$ is separately preserving. In the metric case, by Theorem \ref{thm:sep_sym_preserving}, $F$ is separately and symmetrically preserving.
		
		Let $(X,a)$ be a separated $\mathsf{P}_+^{\,I}$-category (resp. a separated and symmetric $\mathsf{P}_+^{\,I}$-category). Given $i\in I$, each coordinate function $a_i$ of $a$ is an extended quasi-pseudometric on $X$ (resp. an extended pseudometric on $X$).
		
		For every $i\in I$, denote by $(X_i,\overline{a_i})$ the Kolmogorov quotient of $(X,a_i)$, that is, $X_i$ is the quotient set $X/{\sim_i}$, where $\sim_i$ is the following equivalence relation on $X$ 
		\[ x\sim_i y \quad\Longleftrightarrow\quad a_i(x,y)=a_i(y,x)=0, \]
		and $\overline{a_i}:X_i\times X_i\to [0,+\infty]$ is defined by
		\[\overline{a_i}\big([x]_i,[y]_i\big)=a_i(x,y)\qquad\text{ for every }x,y\in X. \]
		Then $\big\{(X_i,\overline{a_i}):i\in I\big\}$ is a family of extended quasi-metric spaces (resp. extended metric spaces). By assumption, $\left(\prod_{i\in I} X_i,F\circ \overline{a_\Pi}\right)$ is an extended quasi-metric space (resp. an extended metric space) and
		\[ \mathscr{T}(F\circ \overline{a_\Pi})\subseteq \prod \mathscr{T}(\overline{a_i}). \]
		
		Since $(X,a)$ is separated, one has
		\[\bigcap_{i\in I}\sim_i=\Delta_X,\]
		and therefore the diagonal mapping
		\[e:X\to \prod_{i\in I}X_i, \qquad e(x)=\big([x]_i\big)_{i\in I},\]
		is injective.
		
		Observe that for every $x,y\in X$,
		\[F\circ \overline{a_\Pi}\big(e(x),e(y)\big)=F\Big(\big(\overline{a_i}([x]_i,[y]_i)\big)_{i\in I}\Big)
		=F\bigl((a_i(x,y))_{i\in I}\bigr)=F\circ a(x,y).\]
		
		Consequently, the topology $\mathscr{T}(F\circ a)$ induced on $X$ is homeomorphic to the subspace topology induced by $\mathscr{T}(F\circ \overline{a_\Pi})$ on $e(X)$. Since 
		\[ \mathscr{T}(F\circ \overline{a_\Pi})\subseteq \prod \mathscr{T}(\overline{a_i}),\]
		it follows that
		\[\mathscr{T}(F\circ a)\subseteq e^{-1}\left(\prod \mathscr{T}(\overline{a_i})\right)\qquad\text{ on }X,\]
		where $e^{-1}\big(\prod \mathscr{T}(\overline{a_i})\big)$ denotes the initial topology on $X$ with respect to $e.$ For any $i\in I$, let $\pi_i:\prod_{j\in I}X_j\to X_i$ be the canonical projection. Since the product topology on $\prod_{i\in I}X_i$ is the initial topology with
		respect to the family $\{\pi_i:i\in I\}$, the initial topology induced by $e$ coincides with the initial topology induced by the family $\{\pi_i\circ e:i\in I\}=\{p_i:i\in I\}$, where $p_i:X\to X_i$ denotes the canonical quotient map.
		Hence, it is the smallest topology on $X$ making continuous all maps $\pi_i\circ e=p_i$, that is, the supremum of the family of initial topologies  $\{p_i^{-1}(\mathscr{T}(\overline{a_i})):i\in I\}.$ Therefore, 
		\[\mathscr{T}(F\circ a)\subseteq \bigvee p_i^{-1}(\mathscr{T}(\overline{a_i}))\qquad\text{ on }X.\]
		Moreover, since $\overline{a_i}(p_i(x),p_i(y))=a_i(x,y)$, the topology
		$p_i^{-1}(\mathscr{T}(\overline{a_i}))$ coincides with the topology induced by $a_i$ on $X$. Indeed,
		\[p_i^{-1}\bigl(B_{\overline a_i}([x]_i,\varepsilon)\bigr)
		=B_{a_i}(x,\varepsilon),\]
		and therefore
		\[p_i^{-1}(\mathscr T(\overline a_i))=\mathscr T(a_i).\]
		Then 
		\[\mathscr{T}(F\circ a)\subseteq \bigvee p_i^{-1}(\mathscr{T}(\overline{a_i}))=\bigvee \mathscr{T}(a_i)=\mathscr{T}(a)\qquad\text{ on }X,\]
		where in the last equality we have used Example \ref{ex:product}.(3). Therefore, $F$ is upper strongly separately preserving (resp. upper strongly separately and symmetrically preserving).
	\end{proof}

	\begin{corollary}[\mbox{cf. \cite[Theorem 3.8]{FBRL26a}}]    \label{cor:charac_upper_eqm}
		Let $F:[0,+\infty]^I\to  [0,+\infty]$ be an extended quasi-metric aggregation function on products (resp. an extended metric aggregation function on products). The following statements are equivalent:
		\begin{enumerate}[(1)]
			\item $F$ is an upper strongly extended quasi-metric aggregation function on products 
			(resp. an upper strongly extended metric aggregation function on products);
			\item $F$ is continuous at $0_I$, when $[0,+\infty]$ is endowed with the left-order topology.
		\end{enumerate}
		
		Moreover, in the quasi-metric case, the previous statements are also equivalent to:
		\begin{enumerate}[(3)]
			\item $F(0_I)=\bigvee \bigl\{F(u):u\ll_{\mathsf{P}_+^I}0_I\bigr\}.$
		\end{enumerate}
	\end{corollary}

	Although our general framework does not yield a characterization of upper strongly extended (quasi-)metric aggregation functions on sets, a result analogous to the previous one can also be obtained in this setting by a direct proof (see \cite{FBRL26a}). This is because, in the metric and quasi-metric cases, aggregation on products and aggregation on sets admit different characterizations \cite{MayorVale10,MayorVale19,MinyaVale19}.

	The first results concerning strongly metric aggregation functions were obtained by Bors\'ik and Dobo\v{s} \cite{BorsikDobos81b,Dobos98}. As in the nonextended metric setting, they considered functions
	\[F:[0,+\infty)^I\to [0,+\infty).\]
	Since $[0,+\infty)$ is not a complete lattice, this framework does not fit directly into our setting. As observed in \cite{FBRL25a}, whenever
	\[F:[0,+\infty)^I\to [0,+\infty)\]
	is a quasi-pseudometric aggregation function on products, the map
	$\overline{F}:\mathsf{P}_+^I\to \mathsf{P}_+$
	defined by
	\[\overline{F}(x)=
	\begin{cases}
		F(x)&\text{if }x_i\neq +\infty\text{ for all }i\in I,\\
		+\infty&\text{otherwise},
	\end{cases}
	\]
	is a preserving function. Analogous statements hold for quasi-metric, pseudometric, and metric aggregation functions on products, yielding separately preserving, symmetrically preserving, and separately and symmetrically preserving functions, respectively. 
	
	So one could expect that the upper strong property of $F$ could be deduced from the same property of $\overline{F}.$ We can prove the following:

	\begin{proposition}\label{prop:upper_Fbar_F}
		Let $F:[0,+\infty)^I\to[0,+\infty).$ Then the following assertions hold.
		\begin{enumerate}[(1)]
			\item If $\overline F$ is an upper strongly extended (quasi-)(pseudo)metric aggregation function on products, then $F$ is an upper strongly (quasi-)(pseudo)metric aggregation function on products.
			\item If $\overline F$ is an upper strongly extended (quasi-)(pseudo)metric aggregation function on sets, then $F$ is an upper strongly (quasi-)(pseudo)metric aggregation function on sets.
		\end{enumerate}
	\end{proposition}
	
	\begin{proof}
		(1) Let $\big\{(X_i,d_i):i\in I\big\}$ be a family of \mbox{(quasi-)}(pseudo)metric spaces. Since each $d_i$ is also an extended (quasi-)(pseudo)metric and
		\[\overline F\circ d_\Pi (x,y)=F\circ d_\Pi (x,y)\]
		for all $x,y\in\prod_{i\in I}X_i$, the upper strongness of $\overline F$ immediately yields
		\[\mathscr T(F\circ d_\Pi)\subseteq\prod_{i\in I}\mathscr T(d_i).\]
		Thus $F$ is an upper strongly (quasi-)(pseudo)metric aggregation function on products.
		
		The proof of (2) is analogous.
	\end{proof}

	Nevertheless, the converse is not true as the next example shows:

	\begin{example}\label{ex:upper_extension_failure}
		Define $F:[0,+\infty)^{\mathbb N}\to[0,+\infty)$
		by
		\[F(u)=\sum_{n=1}^{\infty}2^{-n}\min\{u_n,1\}\qquad\text{ for all }u=(u_n)_{n\in\mathbb{N}}\in [0,+\infty)^{\mathbb N}.\]
		It is well known that $F$ is an upper strongly (quasi-)(pseudo)metric aggregation function on products (see, for instance, \cite[Theorem 4.2.2]{EngelkingBook}). In fact, if $\big\{(X_n,d_n):n\in\mathbb N\big\}$ is a family of \mbox{(quasi-)}(pseudo)metric spaces, then
		\[F\circ d_\Pi(x,y)=F\bigl((d_n(x_n,y_n))_{n\in\mathbb N}\bigr)=\sum_{n=1}^{\infty}2^{-n}\min\{d_n(x_n,y_n),1\},\quad x,y\in \prod_{n\in\mathbb N} X_n,\]
		is a (quasi-)(pseudo)metric on $\prod_{n\in\mathbb N}X_n$ and it induces the product topology.
		
		Let us check that $\overline F$ is not an upper strongly (quasi-)(pseudo)metric aggregation function on products. For each $n\in\mathbb N$, let $X_n=\{0,1\}$ and endow it with the extended metric
		\[d_n(s,t)=\begin{cases}
			0&\text{if }s=t,\\
			+\infty&\text{if }s\neq t.
		\end{cases}
		\]
		Let $0_{\mathbb N}=(0,0,\ldots)\in\prod_{n\in\mathbb N}X_n$. If $y\neq 0_{\mathbb N}$, then there exists $m\in\mathbb N$ such that $y_m\neq 0$, and hence
		\[d_m(0,y_m)=+\infty.\]
		Therefore,
		\[\overline F\big(d_\Pi(0_{\mathbb N},y)\big)=\overline F\bigl((d_n(0,y_n))_{n\in\mathbb N}\bigr)=+\infty.\]
		Consequently,
		\[B_{\overline F\circ d_\Pi}(0_{\mathbb N},1)=\{0_{\mathbb N}\}.\]
		However, $\{0_{\mathbb N}\}$ is not open in the product topology of $\prod_{n\in\mathbb N}X_n.$ Hence
		\[\mathscr T\big(\overline F\circ d_\Pi\big)\nsubseteq\prod_{n\in\mathbb N}\mathscr T(d_n).\]
		Thus $\overline F$ is not an upper strongly extended (quasi-)(pseudo)metric aggregation function on products, although $F$ is an upper strongly (quasi-)(pseudo)metric aggregation function on products.
		
	\end{example}


		
	

	Nevertheless, when $I$ is finite, the converse is true.
	
	\begin{proposition}\label{prop:upper_F_Fbar_finite}
		Let $F:[0,+\infty)^I\to[0,+\infty),$ where $I$ is finite. 
		Then the following assertions hold. 
		\begin{enumerate}[(1)] 
			\item $F$ is an upper strongly (quasi-)(pseudo)metric aggregation function on products if and only if $\overline F$ is an upper strongly extended (quasi-)(pseudo)metric aggregation function on products. 
			\item $F$ is an upper strongly (quasi-)(pseudo)metric aggregation function on sets if and only if $\overline F$ is an upper strongly extended (quasi-)(pseudo)metric aggregation function on sets. 
		\end{enumerate}
	\end{proposition}
	
	\begin{proof}
		(1) The implication from $\overline F$ to $F$ follows from Proposition \ref{prop:upper_Fbar_F}. 
		
		Conversely, assume that $F$ is an upper strongly (quasi-)(pseudo)metric aggregation function on products. A straightforward computation shows that $\overline F$ is an extended (quasi-)(pseudo)metric aggregation function on products.
		
		Let $\big\{(X_i,d_i):i\in I\big\}$ be a finite family of extended (quasi-)(pseudo)metric spaces. We prove that
		\[\mathscr T\big(\overline F\circ d_\Pi\big)\subseteq \prod_{i\in I}\mathscr T(d_i).\]
		Let $x\in\prod_{i\in I}X_i$ and let $\varepsilon>0$. 
		Consider $y\in B_{\overline F\circ d_\Pi}(x,\varepsilon)$. Then $\overline F\circ d_\Pi(x,y)<\varepsilon<+\infty$ so $d_i(x_i,y_i)<+\infty$
		for every $i\in I$. For each $i\in I$, choose $M_i>d_i(x_i,y_i)$ and define
		\[\rho_i=\min\{d_i,M_i\}.\]
		Then $\rho_i$ is a nonextended (quasi-)(pseudo)metric on $X_i.$  Moreover,
		\[\rho_i(x_i,y_i)=d_i(x_i,y_i)\]
		for every $i\in I$, and therefore
		\[F\circ\rho_\Pi(x,y)=\overline F\circ d_\Pi(x,y)<\varepsilon.\]
		
		Since $F$ is an upper strongly (quasi-)(pseudo)metric aggregation function on products, we have
		\[\mathscr T(F\circ\rho_\Pi)\subseteq \prod_{i\in I}\mathscr T(\rho_i).\]
		Thus there exists a neighbourhood $V$ of $y$ in $\prod_{i\in I}\mathscr T(\rho_i)$ such that
		\[V\subseteq B_{F\circ\rho_\Pi}(x,\varepsilon).\]
		Since each $\mathscr T(\rho_i)$ is contained in $\mathscr T(d_i)$, the set $V$ is also a neighbourhood of $y$ in $\prod_{i\in I}\mathscr T(d_i)$.
		Now put
		\[W=V\cap\bigcap_{i\in I}\pi_i^{-1} \bigl(B_{d_i}(y_i,M_i-d_i(x_i,y_i))\bigr).
		\]
		Since $I$ is finite, $W$ is a neighbourhood of $y$ in $\prod_{i\in I}\mathscr T(d_i)$.
		We claim that
		\[W\subseteq B_{\overline F\circ d_\Pi}(x,\varepsilon).\]
		Indeed, let $z\in W$. For every $i\in I$,
		\[d_i(x_i,z_i)\leq d_i(x_i,y_i)+d_i(y_i,z_i)<M_i.\]
		Hence
		\[\rho_i(x_i,z_i)=d_i(x_i,z_i)\]
		for every $i\in I$. Since $z\in V\subseteq B_{F\circ\rho_\Pi}(x,\varepsilon)$, we obtain
		\[\overline F\circ d_\Pi(x,z)=F\circ\rho_\Pi(x,z)<\varepsilon.\]
		Therefore $z\in B_{\overline F\circ d_\Pi}(x,\varepsilon)$, and so
		\[W\subseteq B_{\overline F\circ d_\Pi}(x,\varepsilon).\]
		
		This proves that every ball of $\overline F\circ d_\Pi$ is open in the product topology. Hence $\overline F$ is an upper strongly extended (quasi-)(pseudo)metric aggregation function on products.
		
		(2) The proof for aggregation functions on sets follows in an analogous way.
	\end{proof}

	\begin{corollary}[\mbox{cf.\cite[Theorems 3.8 and 4.15]{FBRL26a}}]
		Let $n\in\mathbb{N}$ and let $F:[0,+\infty)^n\to [0,+\infty)$ be a (quasi-)pseudometric aggregation function on products. The following statements are equivalent:
		\begin{enumerate}[(1)]
			\item $F$ is an upper strongly (quasi-)pseudometric aggregation function on products;
			\item $F$ is an upper strongly (quasi-)pseudometric aggregation function on sets;
			\item $F$ is continuous at $0_n,$ when $[0,+\infty)$ is endowed with the left-order topology.
		\end{enumerate}
	\end{corollary}
	
	\begin{proof}
		By Proposition \ref{prop:upper_F_Fbar_finite}, since $n$ is finite, $F$ is an upper strongly \mbox{(quasi-)}\-pseudometric aggregation function on products if and only if $\overline F$ is an upper strongly extended (quasi-)pseudometric aggregation function on products. The same proposition gives the analogous equivalence for aggregation functions on sets.
		
		Moreover, again because $n$ is finite, $F$ is continuous at $0_n$, when $[0,+\infty)$ is endowed with the left-order topology, if and only if $\overline F$ is continuous at $0_n$, when $[0,+\infty]$ is endowed with the Scott topology. Indeed, a neighbourhood basis of $0_n$ in $[0,+\infty]^n$ is given by the sets of the form
		\[\prod_{i=1}^n[0,\delta_i),\qquad 0<\delta_i<+\infty\]
		All points of these neighbourhoods have finite coordinates.  Hence, on each of them, $\overline F$ coincides with $F$.
		
		The result now follows by applying the Corollary \ref{cor:upper_strongly_ext_qpm} to $\overline F$.
	\end{proof}
	
	In a similar way, but using Corollary \ref{cor:charac_upper_eqm}, we can prove the following.
	
	\begin{corollary}[\mbox{cf.\cite[Theorems 3.8 and 4.15]{FBRL26a}}]
		Let $n\in\mathbb{N}$ and let $F:[0,+\infty)^n\to [0,+\infty)$ be a (quasi-)metric aggregation function on products. The following statements are equivalent:
		\begin{enumerate}[(1)]
			\item $F$ is an upper strongly (quasi-)metric aggregation function on products;
			\item $F$ is continuous at $0_n,$ when $[0,+\infty)$ is endowed with the left-order topology.
		\end{enumerate}
	\end{corollary}

	\begin{remark}
		In the quasi-pseudometric and quasi-metric cases, the equivalence between the statements of the two previous corollaries and
		\[F(0_n)=\bigwedge \bigl\{F(u):0_n<u\bigr\}\]
		follows from our general framework. Indeed, in these cases the corresponding aggregation functions are, respectively, preserving and separately preserving, and hence isotone. Therefore the characterization in terms of Scott continuity yields the above condition.
		
		In the pseudometric and metric cases, however, this equivalence does not follow directly from our general framework, since symmetrically preserving functions need not be isotone. Nevertheless, it follows by applying \cite[Proposition 2, p. 52]{Dobos98}. Indeed, in each of these cases, $F$ preserves triangle triplets. Thus, whenever $a,b\in[0,+\infty)^n$ and $a\leq b$, the triplet $(a,b,b)$ is triangular. Hence
		\[
		F(a)\leq F(b)+F(b)=2F(b).
		\]
		Therefore, although $F$ need not be isotone, it satisfies a weak monotonicity property which is enough to obtain the equivalence in the finite nonextended setting.
	\end{remark}

	\subsection{Lower strongly preserving functions}
	
	We now turn to the characterization of lower strongly preserving functions.
	
	As in the upper case, preservation alone does not imply lower strong preservation. The following simple example illustrates this fact.

	\begin{example}
		
		Let $F:(\mathscr{V},\preceq_\V,\ast)\to(\mathscr{W},\preceq_\W,\star)$ be a function between two continuous quantales given by $F(v)=\top_\W$ for all $v\in\mathscr{V}$. Assume that $\mathscr{V}$ contains at least two different elements. Then $F$ is a lax morphism of quantales, so it is preserving by \cite[Theorem 3.13]{FBRL25a}.
		
		Now, let $X$ be a set with at least two elements, and consider the $\mathscr{V}$-metric space $(X,a)$ such that $a(x,y)=\bot_\V$ if $x\neq y$ and $a(x,y)=\top_\V$ if $x=y$, for every $x,y\in X.$
		Then $\mathscr{T}(a)$ is the discrete topology on $X$, whereas $\mathscr{T}(F\circ a)$ is the indiscrete topology so $\mathscr{T}(a)\nsubseteq \mathscr{T}(F\circ a).$ Hence $F$ is not (lower) strongly preserving.
	\end{example}
	
	To characterize lower strongly preserving functions, we first recall the following notions.  Given two nonempty sets $X$ and $Y$, and denoting by $2^Y$ the power set of $Y$, a multifunction from $X$ to $Y$, denoted by $h:X\rightrightarrows Y,$ is a map $h:X\to 2^Y$ (see, for example, \cite{BeerBook}).  Moreover, if $X$ and $Y$ are topological spaces, then $h$ is \emph{upper semicontinuous} at $x\in X$ if for every open subset $G$ of $Y$ such that $h(x)\subseteq G,$ there exists an open neighborhood $O$ of $x$ such that $h(o)\subseteq G$ for every $o\in O.$ 
	
	If $X$ and $Y$ are complete lattices endowed with the Scott topology, we say that $h$ is \emph{Scott upper semicontinuous at $x\in X$} if it is upper semicontinuous at $x$ with respect to these topologies.
	
	The following theorem provides a complete characterization of lower strongly preserving functions in terms of the fibre of the top element and Scott upper semicontinuity.
	
	\begin{theorem}\label{thm:lowerstronglypreserving}
		Let $F:(\mathscr V,\preceq_\V,\ast)\to(\mathscr W,\preceq_\W,\star)$ be a preserving function between continuous quantales. Then $F$ is lower strongly preserving if and only if the following two conditions hold:
		\begin{enumerate}[(1)]
			\item $F^{-1}(1_\W)=\{1_\V\}$;
			
			\item the inverse multifunction
			\[F^{-1}:\mathscr W\rightrightarrows \mathscr V\]
			is Scott upper semicontinuous at $1_\W.$
		\end{enumerate}
		
		The same characterization holds for lower strongly separately preserving functions, for lower strongly symmetrically preserving functions, and for the simultaneous separately and symmetrically preserving case.
	\end{theorem}
	
	\begin{proof}
		Suppose that $F$ is lower strongly preserving. We first prove (1).
		
		Suppose, by contradiction, that there exists $u\in\mathscr V\setminus\{1_\V\}$ such that $F(u)=1_\W.$ Since $\mathscr V$ is continuous, then 
		\[1_\V=\bigvee \{v\in\mathscr{V}: v\ll_\V 1_\V\}.\]
		Hence, we can find $\varepsilon\ll_\V 1_\V$ such that $\varepsilon\not\preceq_\V u.$ Consider a set $X=\{x_1,x_2\}$ with two different points and $a:X\times X\to \mathscr{V}$ given by
		\[a(x,y)=\begin{cases}
			u&\text{ if }x\neq y,\\
			1_\V & \text{ if }x=y.
		\end{cases}\]
		Since the quantale is integral, $u\ast u\preceq u$, and hence $(X,a)$ is a $\mathscr V$-metric space (in fact, it is also separated and symmetric).
		Since $F$ is lower strongly preserving, we can find $\delta\ll_\W 1_\W$ such that
		\[B_{F\circ a}(x_1,\delta)\subseteq B_a(x_1,\varepsilon).\]
		Nevertheless, $\delta\ll_\W 1_\W=F(u)=F(a(x_1,x_2))$ so $x_2\in B_{F\circ a}(x_1,\delta)\subseteq B_a(x_1,\varepsilon).$ This implies $\varepsilon\ll_\V a(x_1,x_2)=u,$ and hence $\varepsilon\preceq_\V u,$ contradicting the choice of $\varepsilon.$ This proves (1).
		
		To prove (2), consider the separated and symmetric $\mathscr{V}$-metric space $
		(\mathscr V,a)$ obtained by symmetrizing the residuation $h$ on $\mathscr{V}$ (Example \ref{ex:vcat}), namely,
		\[a(u,v)=h(u,v)\ast h(v,u)\qquad \text{ for every }u,v\in\mathscr{V}.\]
		Then $\operatorname{Im}a(1_\V,\cdot)=\mathscr V.$ Indeed, for every $v\in\mathscr V$,
		\[a(1_\V,v)=h(1_\V,v)\ast h(v,1_\V)=v\ast 1_\V=v.\] 
		
		Let $G$ be a Scott-open subset of $\mathscr V$ containing  $F^{-1}(1_\W)=\{1_\V\}.$ Then there exists $\varepsilon\ll_\V 1_\V$ such that
		$\twoheaduparrow\varepsilon\subseteq G.$ By assumption $\mathscr{T}(a)\subseteq \mathscr{T}(F\circ a)$ so we can find $\delta\in\mathscr{W}$ with $\delta\ll_\W 1_\W$ such that
		\[B_{F\circ a}(1_\V,\delta)\subseteq B_a(1_\V,\varepsilon).\]
		Then $\twoheaduparrow\delta$ is a Scott-open neighbourhood of $1_\W$. We claim that this neighbourhood satisfies the required condition. Let $\eta\in \twoheaduparrow\delta $ and $\nu\in F^{-1}(\eta).$ Then $F(a(1_\V,\nu))=F(\nu)=\eta,$  so $\nu\in B_{F\circ a}(1_\V,\delta)\subseteq B_a(1_\V,\varepsilon).$ Hence $\varepsilon\ll_\V a(1_\V,\nu)=\nu. $ This means that
		\[F^{-1}(\eta)\subseteq \twoheaduparrow \varepsilon\subseteq G,
		\]
		which proves (2).
		
		Conversely, suppose that (1) and (2) hold. Let $(X,a)$ be a $\mathscr{V}$-metric space, $x\in X$ and $\varepsilon\in\mathscr{V}$ such that $\varepsilon\ll_\V 1_\V.$ Then $F^{-1}(1_\W)=\{1_\V\}\subseteq \twoheaduparrow\varepsilon.$ By the Scott upper semicontinuity of $F^{-1}$ at $1_\W$, we can find $\delta\ll_\W 1_\W$ such that if $\nu\in \twoheaduparrow \delta$ then  $F^{-1}(\nu)\subseteq \twoheaduparrow \varepsilon.$
		We claim that  
		\[B_{F\circ a}(x,\delta)\subseteq B_a(x,\varepsilon).\] 
		If $y\in B_{F\circ a}(x,\delta)$, then $\delta\ll_\W F(a(x,y))$, that is,
		\[F(a(x,y))\in \twoheaduparrow\delta.\]
		Hence 
		\[a(x,y)\in F^{-1}\Big(F\big(a(x,y)\big)\Big)\subseteq \twoheaduparrow \varepsilon\] 
		so $y\in B_a(x,\varepsilon).$ Thus every basic $\mathscr T(a)$-neighbourhood of $x$ contains a $\mathscr T(F\circ a)$-neighbourhood of $x$, and hence $\mathscr T(a)\subseteq\mathscr T(F\circ a)$. Consequently, $F$ is lower strongly preserving.
		
		The separated case, the symmetric case, and the simultaneous separated and symmetric case are obtained by restricting accordingly the class of $\mathscr V$-metric spaces under consideration.
	\end{proof}
	
	Separately preserving functions satisfy $F^{-1}(1_\W)=\{1_\V\}$ by \cite[Theorem 3.17]{FBRL25a}.  Therefore, in this setting, condition (1) in the preceding theorem is automatically fulfilled and can be omitted. This yields the following result.
	
	\begin{theorem}\label{thm:lowerstronglyseparatelypreserving}
		Let $F:(\mathscr{V},\preceq_\V,\ast)\to(\mathscr{W},\preceq_\W,\star)$ be a separately preserving function between continuous quantales. Then $F$ is lower strongly separately preserving if and only if the inverse multifunction 
		\[F^{-1}:\mathscr W\rightrightarrows \mathscr V\]
		is Scott upper semicontinuous at $1_\W$.
		
		The result remains valid for functions that are both separately and symmetrically preserving.
	\end{theorem}

	As in the upper case, we now consider the particular case of lower strongly separately preserving functions $F:\mathsf{P}_+^{\,I}\to\mathsf{P}_+$.
	An analogue of Proposition \ref{prop:upper_preserving_equal_upper_eqpm} also holds in this setting, with a proof obtained by a straightforward modification of the argument. Combining this observation with Theorem \ref{thm:lowerstronglyseparatelypreserving}, we obtain the following result.

	\begin{proposition}[\mbox{cf. \cite[Proposition 3.4]{FBRL26a}, \cite[Lemma 1, p. 53]{Dobos98}}]\label{prop:eqmag_implies_stronglyeqmag}
		Every extended (quasi-)metric aggregation function on products is a lower strongly extended (quasi-)metric aggregation function on products, that is, for every family $\big\{(X_i,d_i):i\in I\big\}$ of extended (quasi-)metric spaces, it holds that 
		$$\prod \mathscr{T}(d_i)\subseteq \mathscr{T}(F\circ d_\Pi).$$
		
	\end{proposition}
	
	\begin{proof}
		If $F$ is an extended quasi-metric aggregation function on products, then it is a separately preserving function \cite[Theorem 3.18]{FBRL25a}. By Theorem \ref{thm:lowerstronglyseparatelypreserving}, it suffices to show that $F^{-1}$ is Scott upper semicontinuous at $0.$  
		
		By \cite[Theorem 3.18]{FBRL25a}, $F^{-1}(0)=\{0_I\},$ so let $G$ be a Scott-open neighbourhood of $0_I$ in  $\mathsf P_+^{\,I}$. Then there exists $\varepsilon\ll^{\mathrm{op}} 0_I$ such that
		$\twoheaduparrow\varepsilon\subseteq G.$ By the description of the way-below relation in $\mathsf P_+^{\,I}$, there exists a finite subset $J\subseteq I$ such that
		\[\twoheaduparrow\varepsilon=\left\{u\in [0,+\infty]^I: u_j<\varepsilon_j \text{ for every }j\in J\right\}.\]
		If $J=\varnothing$, then $\twoheaduparrow\varepsilon=[0,+\infty]^I$, and the conclusion follows trivially. Hence, we may assume that $J\neq\varnothing$. Notice that $0<\varepsilon_j<+\infty$ for every $j\in J$.

		For every $j\in J$, define $e^j\in[0,+\infty]^I$ by
		\[e^j_i=
		\begin{cases}
			\varepsilon_j&\text{if }i=j,\\
			0&\text{if }i\neq j.
		\end{cases}\]
		Since $e^j\neq 0_I$ and $F^{-1}(0)=\{0_I\}$, we have $F(e^j)>0$
		for every $j\in J$. Since $J$ is finite, we can choose $\delta>0$ such that
		\[2\delta < F(e^j)\qquad\text{for every }j\in J.\]
		
		We claim that
		\[F^{-1}([0,\delta))\subseteq \twoheaduparrow\varepsilon.\]
		Indeed, let $u\in [0,+\infty]^I$ be such that $F(u)<\delta.$ Suppose, by contradiction, that $u\notin \twoheaduparrow\varepsilon$. Then there exists $j\in J$ such that $u_j\geq \varepsilon_j.$ Hence $e^j\leq u$ coordinatewise. Therefore $(e^j,u,u)$ is a triangle triplet in $\mathsf P_+^{\,I}.$ By \cite[Theorem 3.18]{FBRL25a}, $F$ preserves triangle triplets, so $(F(e^j),F(u),F(u))$
		is a triangle triplet in $\mathsf P_+$. In particular,
		\[F(e^j)\leq F(u)+F(u)<2\delta,\]
		which contradicts the choice of $\delta$. Thus $u\in \twoheaduparrow\varepsilon.$
		Therefore
		\[F^{-1}([0,\delta))\subseteq \twoheaduparrow\varepsilon\subseteq G.\]
		It follows that the inverse multifunction 
		\[F^{-1}:\mathsf P_+\rightrightarrows \mathsf P_+^{\,I}\]
		is Scott upper semicontinuous at $0$. By Theorem \ref{thm:lowerstronglyseparatelypreserving}, $F$ is a lower strongly separately preserving function, so a lower strongly extended quasi-metric aggregation function on products. 
		
		The metric case follows from a similar proof. 
	\end{proof}
	
	The above result suggests the natural question whether this holds in an arbitrary setting, that is, whether every separately preserving function $F:\mathscr V\to \mathscr W$ is a lower strongly separately preserving function. By Theorem \ref{thm:lowerstronglyseparatelypreserving} this means that $F^{-1}$ should be Scott upper semicontinuous at $1_\W.$ However, this is not true as the following example shows.
	
	\begin{example}
		Let $X=\{0\}\cup\{1/n:n\in\mathbb N\}$ be endowed with the topology inherited from the real line. Consider the complete lattice $\mathscr V=\mathcal O(X)$
		of open subsets of $X$, ordered by inclusion. Then $(\mathscr V,\subseteq,\cap)$ is a commutative integral continuous quantale whose unit is $1_\mathscr V=X$ \cite[Examples I-1.7]{GierzHoffKeiLawMisScott}.
		
		Consider also the continuous quantale $\mathscr W=([0,1],\leq,\wedge).$
		
		We next define $g:X\to(0,1]$ as
		\[g(0)=1, \qquad g\big(\tfrac{1}{n}\big)=\tfrac{1}{n} \quad\text{ for all }n\in\mathbb N,\]
		and define
		\[F:\mathcal O(X)\to[0,1]\]
		by
		\[F(U)=1-\sup\{g(x):x\in X\setminus U\},\quad\text{ for all }U\in\mathcal O(X)\]
		with the convention that $\sup\varnothing=0$.
		
		First observe that $F(X)=1.$ Moreover, if $U\neq X$, then $X\setminus U\neq\varnothing$, and therefore $\sup\{g(x):x\in X\setminus U\}>0.$ Hence
		$F(U)<1.$
		Consequently,
		\[F^{-1}(1)=\{X\}=\{1_\mathscr V\}.\]
		
		We next show that $F$ preserves asymmetric triangle triplets. Indeed, for every $U,V\in\mathcal O(X)$, one has
		\[\begin{aligned}
			F(U\cap V)
			&=1-\sup\{g(x):x\in X\setminus(U\cap V)\}\\
			&=1-\sup\{g(x):x\in (X\setminus U)\cup (X\setminus V)\}\\
			&=F(U)\wedge F(V).
		\end{aligned}\]
		Moreover, it is clear that $F$ preserves the order. Thus, if $U\cap V\subseteq W,$ then 
		\[F(U)\wedge F(V)= F(U\cap V) \leq F(W).\]
		Therefore $F$ preserves asymmetric triangle triplets. By \cite[Theorem 3.17]{FBRL25a}, $F$ is separately preserving.
		
		However, the inverse multifunction
		\[
		F^{-1}:[0,1]\rightrightarrows\mathcal O(X)
		\]
		is not Scott upper semicontinuous at $1$. Let us show this. Since $X$ is compact, the set $\{X\}$ is Scott-open in $\mathcal O(X)$. Moreover, $F^{-1}(1)=\{X\}\subseteq\{X\}.$
		
		Suppose, by contradiction, that $F^{-1}$ were Scott upper semicontinuous at $1$. Then there would exist a Scott-open neighbourhood $H$ of $1$ in $[0,1]$ such that
		\[
		F^{-1}(t)\subseteq\{X\}
		\]
		for every $t\in H$.
		
		Every Scott-open neighbourhood of $1$ in $[0,1]$ contains an interval of the form
		$(\alpha,1]$ for some $\alpha<1$. Choose $n\in\mathbb N$ such that
		$\alpha<1-\tfrac{1}{n}.$
		Let
		\[U_n=X\setminus\{\tfrac{1}{n}\}.\]
		Then $U_n$ is open in $X$ and $U_n\neq X$. Moreover,
		\[F(U_n)=1-\tfrac{1}{n}>\alpha.\]
		Hence
		\[F(U_n)\in H.\]
		However,
		\[U_n\in F^{-1}\left(1-\tfrac{1}{n}\right)\qquad\text{and}\qquad U_n\neq X,\]
		which contradicts
		\[F^{-1}(t)\subseteq\{X\}\quad\text{for every }t\in H.\]
		Therefore $F^{-1}$ is not Scott upper semicontinuous at $1$.
	\end{example}
	
	Combining Theorem \ref{thm:lowerstronglypreserving} with the
	argument used in Proposition \ref{prop:eqmag_implies_stronglyeqmag}, we obtain the following characterization.
	
	\begin{proposition}[\mbox{cf.\cite[Theorem 3.7]{FBRL26a}}]\label{prop:lseqpm}
		Let $F:[0,+\infty]^I\to [0,+\infty]$ be an extended (quasi-)pseudometric aggregation function on products. The following statements are equivalent:
		\begin{enumerate}[(1)]
			\item $F$ is a lower strongly extended (quasi-)pseudometric aggregation function on products;
			\item $F$ is a lower strongly extended (quasi-)pseudometric aggregation function on sets;
			\item $F^{-1}(0)=\{0_I\}.$
		\end{enumerate}
	\end{proposition}
	
	\begin{proof}
		The equivalence between (1) and (2) follows from the lower strongly analogue of Proposition \ref{prop:upper_preserving_equal_upper_eqpm}.
		
		Moreover, (2) implies (3) by Theorem \ref{thm:lowerstronglypreserving}.
		
		Finally, the implication (3) $\Rightarrow$ (1) follows by repeating the argument in the proof of Proposition \ref{prop:eqmag_implies_stronglyeqmag}, using condition (3) to obtain the required fibre condition.   
	\end{proof}

	As we already mentioned, Bors\'ik and Dobo\v{s} characterized strongly metric aggregation functions on products. In the upper setting, our framework does not directly yield their results in an arbitrary Cartesian product (see Example \ref{ex:upper_extension_failure}). More precisely, the difficulty arises from the fact that given $F:[0,+\infty)^I\to [0,+\infty)$, then its extension    $\overline{F}:\mathsf{P}_+^I\to \mathsf{P}_+$
	defined by
	\[\overline{F}(x)=
	\begin{cases}
		F(x)&\text{if }x_i\neq +\infty\text{ for all }i\in I,\\
		+\infty&\text{otherwise},
	\end{cases}
	\]
	fails to be Scott continuous at $0_I$ whenever $I$ is infinite. 
	
	Nevertheless, in the lower setting, we can prove the following:
	
	\begin{proposition}\label{prop:lower_F_Fbar}
		Let $F:[0,+\infty)^I\to[0,+\infty).$
		Then $F$ is a lower strongly (quasi-)(pseudo)metric aggregation function on products if and only if $\overline F$ is a lower strongly extended (quasi-)(pseudo)metric aggregation function on products.
	\end{proposition}
	
	\begin{proof}
		Assume first that $\overline F$ is a lower strongly extended (quasi-)(pseudo)metric aggregation function on products. Let $\big\{(X_i,d_i):i\in I\big\}$ be a family of \mbox{(quasi-)}(pseudo)metric spaces. Since each $d_i$ is also an extended (quasi-)(pseudo)metric and
		\[
		\overline F\circ d_\Pi (x,y)=F\circ d_\Pi (x,y)\]
		for all $x,y\in\prod_{i\in I}X_i$, the lower strongness of $\overline F$ immediately yields
		\[\prod_{i\in I}\mathscr T(d_i)\subseteq \mathscr T(F\circ d_\Pi).\]
		Thus $F$ is a lower strongly (quasi-)(pseudo)metric aggregation function on products.
		
		Conversely, assume that $F$ is a lower strongly (quasi-)(pseudo)metric aggregation function on products. A straightforward computation shows that $\overline F$ is an extended (quasi-)(pseudo)metric aggregation function on products.
		
		Let $\big\{(X_i,d_i):i\in I\big\}$ be a family of extended (quasi-)(pseudo)metric spaces.
		
		We show that
		\[\prod_{i\in I}\mathscr T(d_i)\subseteq \mathscr T(\overline F\circ d_\Pi).\]
		
		Let $x\in\prod_{i\in I}X_i$ and let
		\[U=\bigcap_{j\in J}\pi_j^{-1}\bigl(B_{d_j}(x_j,\varepsilon_j)\bigr)\]
		be a basic product neighbourhood of $x$, where $J\subseteq I$ is finite and $\varepsilon_j>0$ for each $j\in J$. If $J=\varnothing$, there is nothing to prove. 
		
		For each $i\in I$, choose $M_i>0$ in such a way that $M_j>\varepsilon_j$ for every $j\in J$, and define
		\[\rho_i=\min\{d_i,M_i\}.\]
		Then $\rho_i$ is a nonextended (quasi-)(pseudo)metric on $X_i$. Moreover, for every $j\in J$ we have
		\[B_{\rho_j}(x_j,\varepsilon_j)=B_{d_j}(x_j,\varepsilon_j),\]
		because $M_j>\varepsilon_j$. Hence
		\[U=\bigcap_{j\in J}\pi_j^{-1}\bigl(B_{\rho_j}(x_j,\varepsilon_j)\bigr).\]
		Since $F$ is a lower strongly (quasi-)(pseudo)metric aggregation function on products, there exists $\delta>0$ such that
		\[B_{F\circ\rho_\Pi}(x,\delta)\subseteq U.\]
		We claim that
		\[B_{\overline{F}\circ d_\Pi}(x,\delta/2)\subseteq U.\] 
		Let $y\in B_{\overline{F}\circ d_\Pi}(x,\delta/2)$ and put
		\[a=(d_i(x_i,y_i))_{i\in I}.\]
		Then $\overline F(a)=\overline{F}\circ d_\Pi(x,y)<\delta/2<+\infty.$
		Therefore $a\in[0,+\infty)^I$.
		Now define
		\[
		b=(\rho_i(x_i,y_i))_{i\in I}.
		\]
		Since  $(b,a,a)$ is a triangle triplet and $F$ preserves triangle triplets \cite{Dobos98,MayorVale10,FBRL25a}
		\[(F\circ\rho_\Pi)(x,y)=F(b)\leq F(a)+F(a)=2F(a)=2\overline F(a)<\delta.\]
		Thus
		\[y\in B_{F\circ\rho_\Pi}(x,\delta)\subseteq U.\]
		Hence
		\[
		B_{\overline F\circ d_\Pi}(x,\delta/2)\subseteq U.
		\]
		
		Consequently, $\overline F$ is a lower strongly extended (quasi-)(pseudo)metric aggregation function on products.
	\end{proof}

	Therefore, we can characterize lower strongly (quasi-)(pseudo)metric aggregation functions on products by means of our general framework after passing to the extension $\overline F$. In order to obtain an intrinsic characterization in terms of $F$, we use the following observation.
	
	\begin{proposition}\label{prop:F_Fbar_upper_semicontinuous}
		Let $F:[0,+\infty)^I\to[0,+\infty)$ be a map. 
		Endow $[0,+\infty)$ with its left-order topology and $[0,+\infty)^I$ with the corresponding product topology. Then $F^{-1}$
		is upper semicontinuous at $0$ if and only if $\overline F^{-1}$ is Scott upper semicontinuous at $0$.
	\end{proposition}

	\begin{proof}
		It is obvious that
		\[\overline F^{-1}(0)=F^{-1}(0).\]
		Assume first that $F^{-1}$ is upper semicontinuous at $0$. Let $G$ be an open subset of $[0,+\infty]^I$ such that
		\[\overline F^{-1}(0)\subseteq G.\]
		Then $G\cap[0,+\infty)^I$ is an open subset of $[0,+\infty)^I$ containing $F^{-1}(0)$. By the upper semicontinuity of $F^{-1}$ at $0$, there exists a basic open neighbourhood $[0,\delta)$ of $0$ in $[0,+\infty)$ such that
		\[F^{-1}(t)\subseteq G\cap[0,+\infty)^I\]
		for every $t\in [0,\delta)$. 
		Given $t\in[0,\delta)$ then $\overline F^{-1}(t)=F^{-1}(t)$ so
		\[\overline F^{-1}(t)\subseteq G.\]
		Thus $\overline F^{-1}$ is Scott upper semicontinuous at $0$.
		
		Conversely, assume that $\overline F^{-1}$ is Scott upper
		semicontinuous at $0$. Let $H$ be an open subset of
		$[0,+\infty)^I$, endowed with the product left-order topology, such that $F^{-1}(0)\subseteq H.$ Since this topology is the subspace topology induced by the Scott topology of $[0,+\infty]^I$, there exists a Scott-open set $G\subseteq[0,+\infty]^I$ such that
		\[H=G\cap[0,+\infty)^I.\]
		Moreover,
		\[\overline F^{-1}(0)=F^{-1}(0)\subseteq H\subseteq G.\]
		By assumption, there exists a basic Scott-open neighbourhood $[0,\delta)$ of $0$ in $[0,+\infty]$ such that
		\[\overline F^{-1}(t)\subseteq G\]
		for every $t\in [0,\delta)$. Then
		\[F^{-1}(t)=\overline F^{-1}(t)\cap[0,+\infty)^I\subseteq G\cap[0,+\infty)^I=H, \]
		for every $t\in[0,\delta).$ Hence $F^{-1}$ is upper semicontinuous at $0$.
	\end{proof}
	
	We conclude with the following consequences of the preceding results.
	
	The following results are direct consequences of Propositions \ref{prop:lower_F_Fbar} and \ref{prop:F_Fbar_upper_semicontinuous}.

	\begin{corollary}[\mbox{cf. \cite[Proposition 3.4]{FBRL26a}, \cite[Lemma 1, p. 53]{Dobos98}}]
		Every \mbox{(quasi-)}\-metric aggregation function on products is a lower strongly (quasi-)metric aggregation function on products, that is, for every family $\big\{(X_i,d_i):i\in I\big\}$ of (quasi-)metric spaces, it holds that 
		$$\prod \mathscr{T}(d_i)\subseteq \mathscr{T}(F\circ d_\Pi).$$
	\end{corollary}

	\begin{corollary}[\mbox{cf.\cite[Theorem 3.7]{FBRL26a}}]
		Let $F:[0,+\infty)^I\to [0,+\infty)$ be a (quasi-)pseudometric aggregation function on products. The following statements are equivalent:
		\begin{enumerate}[(1)]
			\item $F$ is a lower strongly  (quasi-)pseudometric aggregation function on products;
			\item $F$ is a lower strongly  (quasi-)pseudometric aggregation function on sets;
			\item $F^{-1}(0)=\{0_I\}.$
		\end{enumerate}
	\end{corollary}
	
	\begin{proof}
		By Proposition \ref{prop:lower_F_Fbar}, $F$ is a lower strongly
		(quasi-)pseudometric aggregation function on products if and only if $\overline F$ is a lower strongly extended (quasi-)pseudometric
		aggregation function on products. By Proposition \ref{prop:lseqpm},
		this is equivalent to
		\[\overline F^{-1}(0)=\{0_I\}.\]
		Since
		\[\overline F^{-1}(0)=F^{-1}(0),\]
		we obtain the equivalence between (1) and (3).
		
		Moreover, in the (quasi-)pseudometric setting, lower strong aggregation on products and on sets are equivalent (\cite[Theorem 3.16]{FBRL25a}, \cite{PraTri02}). Hence (1) and (2) are equivalent, and the result follows.
	\end{proof}

	\subsection{Strongly preserving functions}
	
	Our previous results yield the following characterization of strongly preserving functions.
	
	\begin{theorem}\label{thm:strongly_preserving}
		Let $F:(\mathscr{V},\preceq_\V,\ast)\to(\mathscr{W},\preceq_\W,\star)$ be a preserving function between continuous quantales. Then $F$ is strongly preserving if and only if the following statements hold:
		\begin{enumerate}[(1)]
			\item $F$ is Scott continuous at $1_\V;$
			\item $F^{-1}(1_\W)=\{1_\V\};$
			\item the inverse multifunction   
			\[F^{-1}:\mathscr W\rightrightarrows \mathscr V\]
			is Scott upper semicontinuous at $1_\W.$
			
		\end{enumerate}
		The same characterization holds in the separately preserving,
		symmetrically preserving, and simultaneously separately and
		symmetrically preserving settings.
	\end{theorem}
	
	\begin{proof}
		By definition, $F$ is strongly preserving if and only if it is both upper strongly preserving and lower strongly preserving. The result follows by combining Theorem \ref{thm:us} with Theorem \ref{thm:lowerstronglypreserving}. The corresponding statements in the separated, symmetric, and separated and symmetric settings follow in the same way from the respective upper and lower versions of these results.
	\end{proof}

	Combining the corresponding upper and lower results yields concrete characterizations of strongly extended (quasi-)(pseudo)metric aggregation functions. For instance, Corollary
	\ref{cor:upper_strongly_ext_qpm} and Proposition \ref{prop:lseqpm}
	give the following result.

	\begin{corollary}[\mbox{cf. \cite[Theorem 4.17]{FBRL26a}}]\label{cor:strongly_ext_qpm}
		Let $F:[0,+\infty]^I\to [0,+\infty]$ be an extended quasi-pseudometric aggregation function on products 
		(resp. an extended pseudometric aggregation function on products). The following statements are equivalent:
		\begin{enumerate}[(1)]
			\item $F$ is a strongly extended quasi-pseudometric aggregation function on products 
			(resp. a strongly extended pseudometric aggregation function on products);
			
			\item $F$ is a strongly extended quasi-pseudometric aggregation function on sets 
			(resp. a strongly extended pseudometric aggregation function on sets);
			
			\item $F^{-1}(0)=\{0_I\}$ and $F$ is continuous at $0_I$, when $[0,+\infty]$ is endowed with the left-order topology.
		\end{enumerate}
		
	\end{corollary}
	
	Further characterizations for the remaining classes of aggregation
	functions considered above can be obtained analogously by combining the corresponding upper and lower results.
	
	\section{Conclusions}
	
	In this paper, we have analysed when aggregation preserves the topology associated with quantale-valued metric spaces. To address this problem, we introduced upper and lower strong preservation, corresponding to the two possible inclusions between the topology of the original space and the topology induced after aggregation. The characterizations obtained for these two properties involve different conditions, and together they determine exactly when aggregation preserves the induced topology.
	
	We then considered functions from $\mathsf P_+^{\,I}$ to $\mathsf P_+$. This particular case yields characterizations of topology-preserving aggregation for several classes of extended generalized metrics, including quasi-pseudometrics, pseudometrics, quasi-metrics, and metrics.
	
	The framework developed here is not restricted to generalized metric structures valued in $\mathsf P_+$. It can also be applied to other quantale-valued structures considered in aggregation theory. An important class of examples is provided by fuzzy metric-type structures \cite{FBRL25a}. The corresponding topological preservation problem for such structures will be addressed in future work.

	

\end{document}